\documentclass[final,11pt,letterpaper,leqno]{siamltex}
\usepackage{amssymb}
\usepackage{amsmath}
\usepackage{color}
\usepackage{graphicx}
\graphicspath{{.}}

\newtheorem{thm}{Theorem}[section]
\newtheorem{lem}{Lemma}
\newtheorem{defn}{Definition}[section]

\newtheorem{remark}{Remark}

\newcommand{\norm}[1]{\Vert#1\Vert}

\newcommand{\abs}[1]{\vert#1\vert}
\newcommand{\Abs}[1]{\left\vert#1\right\vert}

\newcommand{\bq}{\begin{equation}}
\newcommand{\eq}{\end{equation}}
\newcommand{\R}{\mathbb{R}}

\newcommand{\e}{\epsilon}
\newcommand{\bO}{\mathcal{O}}
\newcommand{\dm}{\,d}

\newcommand{\al}{\alpha}

\newcommand{\FL}{(-\Delta)^{\al/2}}
\newcommand{\FLh}{(-\Delta_h)^{\al/2}}
\newcommand{\hhat}{T_h}
\newcommand{\Qe}{Q_h}
\newcommand{\Qo}{R_h}
\newcommand{\Dd}{D}

\newcommand{\Ddh}{\Dd_h}
\newcommand{\DdCh}{\Dd^\mathbf{C}_h}

\newcommand{\Cal}{{C_{1,\al}}}
\newcommand{\meas}{\nu}

\title{Numerical methods for the fractional {L}aplacian: \\
  a finite difference-quadrature approach}
\author{Yanghong Huang\thanks{Department of Mathematics
Imperial College London, London SW7 2AZ, United Kingdom. Email: yanghong.huang@imperial.ac.uk} 
\and Adam Oberman\thanks{Department of Mathematics and Statistics, McGill University, Montreal, Canada. Email: adam.oberman@mcgill.ca}}

\date{\today}
\begin{document}

\maketitle

\begin{abstract}
The fractional Laplacian $(-\Delta)^{\alpha/2}$ is a non-local operator which depends on the parameter $\al$ and recovers the usual Laplacian as $\al \to 2$.  
A numerical method for the fractional Laplacian is proposed, based on the singular integral representation for the operator.
The method combines finite differences with numerical quadrature, to obtain a discrete convolution operator with positive weights.
The accuracy of the method is shown to be $O(h^{3-\al})$.  Convergence of the method is proven.   
The treatment of  far field boundary conditions using an asymptotic approximation to the integral is used to obtain 
an accurate method.
Numerical experiments on known exact solutions validate the predicted convergence rates.  Computational examples include exponentially and algebraically decaying solutions with varying regularity.  The generalization to nonlinear equations involving the operator is discussed: the obstacle problem for the fractional Laplacian is computed.

\end{abstract}



\pagestyle{myheadings}
\thispagestyle{plain}
\markboth{Y. Huang and A. Oberman}{Numerical methods for the fractional {L}aplacian}

\section{Introduction}

The fractional Laplacian is the prototypical operator to model non-local diffusions.  These non-local (or anomalous) diffusions, which incorporate long range interactions, have been the subject of much interest in recent years ~\cite{MR0361633,MR1890104,MR2807926,herrmann2011fractional}.
Despite the diversity of the numerical methods in use, the modern tools of numerical analysis have not been applied to the study of the operator.  In particular, the accuracy of the methods may be unknown, and convergence proofs have not been firmly established, for the majority of the relevant numerical methods.  

In this paper, we derive a finite difference/quadrature  method for the \emph{fractional Laplacian}.  We obtain the accuracy of the method (at least on smooth solutions) and prove the convergence of the method, in the setting of the extended Dirichlet problem.   We also include supplemental materials (excerpted from~\cite{HuangOberman2}, in preparation) where we perform a comparative numerical study of this methods with existing methods.   We make an additional effort to study the effect of truncation of the domain on the accuracy of the solution, and derive an asymptotic approximation which can significantly reduce the error from truncation.

On the entire space $\R^n$, the fractional Laplacian $(-\Delta)^{\alpha/2}$ with $0 < \alpha <2$ can be defined in many equivalent ways.
 It is defined simply as a pseudo-differential operator with symbol $|\xi|^\alpha$~\cite{MR0290095}, that is, via the Fourier transform $\mathcal{F}$,
\bq\label{eq:fftfrac}
\mathcal{F}\big[\FL u\big](\xi) = |\xi|^{\al}\mathcal{F}[u](\xi).
\eq
The definition~\eqref{eq:fftfrac} provides a simple method for solving $\FL u = f$  (assuming  $f$ decays quickly enough at infinity).  
An equivalent definition of the operator is given by  a singular integral~\cite{MR0350027},
\bq\label{eq:rieszfrac}
 \FL u(x) = C_{n,\al} \, \int_{\mathbb{R}^n}
\frac{u(x)-u(y)}{|x-y|^{n+\al}} dy.
\eq
Here the constant $C_{n,\alpha}$ is given by 
 \begin{equation}\label{eq:constC}
 C_{n,\al} = \frac{\al 2^{\al-1}\Gamma\left(\frac{\al+n}{2}\right)}
{\pi^{n/2}\Gamma\left(\frac{2-\al}{2}\right)},
\end{equation}
and $\Gamma(x)$ is the Gamma function. In probability theory, the fractional
Laplacian is the infinitesimal generator of symmetric $\alpha$-stable L\'{e}vy
process~\cite{MR2512800}. In financial mathematics, it appears as an alternative model to Brownian motion, incorporating the jumps in asset prices~\cite{cont2004financial,raible2000levy}. These representations are shown to be  equivalent in~\cite{MR0350027,MR0290095,Vald11}.

On a bounded domain, the study of the operator becomes more complicated. 
In contrast to the case of the (standard) Laplacian operator,  where Dirichlet and Neumann boundary conditions are well understood, and have simple interpretations at the particle or probabilistic level, physical or probabilistically motivated interpretations of the fractional Laplacian operator on bounded domains are not well established \cite{chechkin2003first,FLBD}. 
The different representations of the fractional Laplacian may lead to different operators when restricted to a bounded domain, and this provided challenges for numerical methods, which naturally require truncation of the operator to a bounded domain. 
This simplest case is for periodic boundary 
conditions, where  the operator can still be defined via~\eqref{eq:fftfrac}. On an interval $[-L,L]$, the fractional Laplacian can be defined in terms of the left and right Riemann-Liouville fractional derivatives ${}_{-L}D_x^\alpha u(x)$ and ${}_xD_L^\al u(x)$~\cite{MR1347689,gorenflo1998random}, 
\begin{equation}\label{eq:fracdevlap}
(-\Delta)^{\alpha/2}u(x) = 
\frac{{}_{-L}D_x^\alpha u(x)+{}_xD_L^\al u(x)}{2\cos (\alpha
\pi/2)}, \quad \alpha \neq 1.
\end{equation}
The  spectral decomposition gives another way to define the fractional Laplacian  on a bounded domain $\Dd \subset \R^n$.
Let $(\lambda_k, \phi_k)_k$ be the eigenpairs of the (negative) Laplacian operator $-\Delta$,
\begin{equation*}\label{eq:eigexpan}
-\Delta \phi_k = \lambda_k \phi_k,
\end{equation*}
subject to appropriate boundary conditions which ensure that all the  $\lambda_k$ 
are nonnegative and and that $\{\phi_k\}$ is a complete orthonormal basis. 
Then if $u$ has the expansion $u(x) = \sum_{k} c_k \phi_k(x)$, we can define
\begin{equation*}\label{eq:spectraldecom}
(-\Delta)^{\alpha/2} u(x) = \sum_{k} c_k\lambda_k^{\alpha/2}\phi_k(x).
\end{equation*}
However, it is not clear how to interpret the spectral definition at the particle level or, more rigorously, in terms of the underlying L\'{e}vy process.  On the other hand, most of these definitions formally converge to the fractional Laplacian operator in $\mathbb{R}^n$ as the domain is extended to the whole space. 

While the appropriate treatment of different boundary conditions is still an open problem,  a natural  starting point  is the extended Dirichlet boundary value problem, given by
\bq\label{FLD}\tag{FLD}
\begin{aligned}
\FL u &= f, && \text{ for } x \in  \Dd
\\
u &= g,  && \text{ for  } x \in \R^n\!\setminus\! \Dd.
\end{aligned}
\eq
Here the functions $f$ and $g$ are given, with appropriate smoothness and decaying conditions. 
Note that now~\eqref{eq:fftfrac} cannot be applied directly, since $f$ is defined only on the bounded domain $D$.

Using the singular integral definition~\eqref{eq:rieszfrac}, the operator splits into
\[ 
\FL u(x) = C_{n,\al} \left( 
\int_{\Dd}  \frac{u(x)-u(y)}{|x-y|^{n+\al}} dy + 
\int_{\mathbb{R}^n\setminus\Dd}  \frac{u(x)-g(y)}{|x-y|^{n+\al}} dy 
\right), \quad x\in \Dd
\]
where the representation   now makes it clear that the unknowns are only in $\Dd$, despite the fact that we have an integral over $\R^n$.  In contradistinction to the Dirichlet problem for the standard Laplacian,
the values of $g$ 
are  required on the entire complement of the domain $\Dd$, rather than only on 
 boundary $\partial \Dd$.   
Solving the extended Dirichlet problem~\eqref{FLD} is a building block for the treatment of more general nonlinear problems, for example, the obstacle problem~\cite{silvestre2007regularity}, which we touch upon below.

In the special case $g=0$, $f=1$, the solution, $u(x)$, of \eqref{FLD} is related to the probability density function of the first exit time of the symmetric  $\alpha$-stable L\'{e}vy process from a 
given domain $\Dd$~\cite{MR0137148}. See also \cite{BurchLehoucq} where related problems with volume constraints are discussed.
In the case where $\Dd$ is a ball about the origin,  there is 
an explicit expression (``balayage problem'') for the solution, which is an integral involving $f$ and $g$ (\cite[Chapter I]{MR0350027} or~\cite[Section 5.1]{silvestre2007regularity}).  The balayage integral generalizes the classical Poisson formula for the Laplace operator.

More general Parabolic Integro-Differential Equations (PIDE) driven by L\'evy processes appear in mathematical finance~\cite{cont2005finite}.  
The  infinitesimal generators $L^X$ of  L\'evy processes are of the form 
\[
L^{X}u(x) = \frac{\sigma^2}{2}u_{xx}(x) + \gamma u_x(x) + I[u](x)
\]
where the nonlocal operator reads 
\begin{equation}\label{eq:PIDEnonlocal}
I[u](x) = \int_{\mathbb{R}} \left(\vphantom{e^{\frac{3}{2}}} u(x+y) - u(x) - y\chi_{\abs{y}\le 1}(y) u_x(x) \right)   \,  \meas(y)\,dy. 
\end{equation}
The fractional Laplacian operator is recovered when $\meas(y) = C_{1,\alpha}\abs{y}^{-1-\al}$.  Part of our derivation applies to a general L\'{e}vy measure $\meas(y)\,dy$, although our focus is on the 
special case of the fractional Laplacian.  

Many numerical methods have been proposed to solve equations involving the fractional Laplacian operator, either on the whole space or bounded domains. A majority of them are related to fractional derivatives (either Riemann-Liouville or Caputo type) on bounded domains, termed  \emph{fractional diffusion}, as summarized in~\cite{diethelm2005algorithms}. 
The theoretical aspects of equations involving fractional derivatives are not clearly established.
In contrast, during the last decade equations involving the fractional Laplacian on
the whole space have been studied intensively: for instance the fractional Burgers 
equation~\cite{MR1637513,MR2227237} and fractional porous medium equations~\cite{MR2737788,MR2817383,MR2847534}.

After we submitted this work, we became aware of recent work on nonlocal diffusion problems and applications in the context of peridynamics.  We refer to the  review article \cite{du2012analysis}.   In the reference  \cite{tian2013analysis} a finite difference/quadrature method  is used to approximate a truncated nonlocal diffusion operator.  Convergence of the approximation is also proved, using similar maximum principle techniques.    In \cite{d2013fractional}, a study of convergence of approximations of the fractional Laplacian operator is performed, paying particular attention to the influence of the truncation size of the operator and the computational domain size.

Compared with the advances in theoretical analysis, there are not many state-of-the-art numerical methods designed for equations involving the fractional Laplacian. Surprisingly, the nonlinear theory is ahead of the linear theory, in the sense that once a suitable numerical method is obtained for the fractional Laplacian, it can be extended to nonlinear problems.  
For example, in the case of nonlinear elliptic operators involving the fractional Laplacian, a necessary requirement for convergence to the unique viscosity solution of the equation is that the scheme is consistent and monotone~\cite{biswas2010difference}~\cite{alvarez1996viscosity}.   In this context, for a linear equation,  monotonicity corresponds to positive weights.

Similarly, for the fractal conservation laws considered by Droniou~\cite{MR2552219},  a suitable scheme for the  fractional Laplacian allows a nonlinear  numerical method to be built which converges to the  corresponding entropy solution. 
In a similar spirit, Cifani, Jakobsen and their colleagues~\cite{MR2795714,MR2832791,MR2855429} provide results for other equations with convection or degenerate diffusion. However, the focus of these schemes is more on the right limiting entropy solutions, while the accuracy and the fast implementation of the schemes may not be essential. 
In financial mathematics, a discussion of the related numerical methods can be found in~\cite[Chapter 12]{cont2004financial}.   Numerical methods for those parabolic obstacle Integro-Differential Equations have been studied in~\cite{cont2005finite},  including a proof of convergence to the viscosity solution. Our methods are obtained using similar methods to~\cite{
cont2005finite}, but with improvements in the order of accuracy and the treatment of boundary conditions.


In theory, the fractional Laplacian operator can be approximated numerically by any of the different yet equivalent definitions above, \eqref{eq:fftfrac}, \eqref{eq:rieszfrac} or \eqref{eq:fracdevlap}. The spectral methods based on~\eqref{eq:fftfrac}, usually implemented with FFT, are effective for periodic domains. However in the whole space,  the slow decay of the solutions requires a large number of modes or introduces significant aliasing errors. Schemes based on the
integral representation must be designed carefully to avoid large errors near the singularity and to take into account the contribution on the unbounded space. The popular Gr\"unwald-Letnikov type difference methods based on~\eqref{eq:fracdevlap}  become singular when $\alpha\approx 1$~\cite{Tadjeran2006205}, already observed from the denominator in their definition.  In addition, there is no isotropic extension of the fractional derivatives to higher dimensions. 
An alternative to directly representing the operator, is to use an extension problem~\cite{caffarelli2007extension}, which results in a local problem on a higher dimensional space.  Numerical schemes using the extension problem have been built,  based on finite difference based methods~\cite{Felix,JLFelix} and finite element methods~\cite{nochetto2013pde}.

In this paper, we derive a finite difference/quadrature discretization of the fractional Laplacian based on the singular integral definition~\eqref{eq:rieszfrac} in one dimension.  The operator is  approximated by a discrete convolution of the function grids values $u_i$ with positive weights $w_j$, defined for $j = 1, \dots,  \infty$.  We write the discrete operator as 
\bq
\label{FLh}\tag{$\text{FL}_h$}
    (-\Delta_h)^{\alpha/2} u_i 
    = \sum_{j=1}^\infty     (2u_i-u_{i+j}- u_{i-j})w_j
    =  \sum_{j=-\infty}^\infty  (u_i-u_{i-j})w_j.
\eq
In the second expression, we have extended the weights symmetrically, $w_{-j} = w_j$  (there is no need to define $w_0$, since this term makes no contribution to the sum).
The positivity and decay of the kernel of the fraction Laplacian operator are inherited by the discrete fractional Laplacian operator. 
The discrete operator can be computed efficiently using a fast convolution algorithm.

The discretization is derived in Section~\ref{sec:weights}, with special attention to the singularity in the integral.  The explicit weights $w_j$ 
are given in Section~\ref{sect:explicitweights}. 
The convergence proof of the extended  Dirichlet problem is given in Section~\ref{sec:proof}. The treatment of the far field boundary conditions is performed in Section~\ref{sec:farfieldBC}.
Validation of the accuracy and  numerical convergence tests are performed in 
Section~\ref{sec:numerics}.

\section{ Quadrature and finite difference discretization of the fractional Laplacian} \label{sec:weights}

In this section we derive the combined finite difference/exact quadrature  discretization of the fractional Laplacian operator in one dimension. 
The weights, $w_j$, in the convolution~\eqref{FLh} are collected from 
approximations of the integral~\eqref{eq:rieszfrac}, by splitting it into two parts.  In the singular part of the integral, we obtain a (rescaled) second derivative, it is discretized using a standard centered finite difference.  In the tail of the integral, because a direct quadrature method would lead to large errors, we  perform a non-standard quadrature.  This semi-exact quadrature uses exact integration of the weight function $\nu(y)$, 
multiplied by an interpolation of the unknown function $u$.

\subsection{The singular integral operator}

We consider the slightly more general singular integral operator
\bq\label{eq:SI}\tag{SI}
  I^\meas[u](x) = \mathrm{P.V.}
    \int_{\R}  \big(u(x)-u(x-y)\big) \meas(y)\, dy,
\eq
where $\mathrm{P.V.}$ indicates a principal value integral.
The nonnegative measure $\meas(y) \,dy$ satisfies the conditions: 
\begin{equation}\label{eq:levymeas}
\int_{-1}^1 y^2\nu(y)dy<\infty, \quad \int_{|y|>1}\nu(y)dy <\infty.
\end{equation}
In the special case 
\[
\meas(y) = \meas^\al(y) \equiv \Cal\abs{y}^{-1-\al},
\]
 the definition~\eqref{eq:SI} recovers the singular integral representation~\eqref{eq:rieszfrac} of
the fractional Laplacian.

We immediately divide the integral~\eqref{eq:SI} into two parts, the singular part, and the tail.
\begin{defn}
Take $h < 1$ and  define the singular part, and the tail, of the integral to be, respectively
\begin{align}
\label{IsingDefn}
I^\meas_{S}[u](x) = \mathrm{(P.V.)} & \int_{\abs{y} \le h}
  \big(u(x)-u(x-y)\big) \meas(y) \dm y,
\\
\label{Tail}
 I^\meas_{T}[u](x) = &\int_{\abs{y} > h}
  \big(u(x)-u(x-y)\big) \meas(y) \dm y.
\end{align}
\end{defn}

Next we will specialize to functions on the grid, $\mathbb{Z}_h  = \{ x_i= ih \mid i= 0,\pm 1, \pm 2, \dots  \}$,
with uniform spacing, $h$, and derive an expression for the discrete
fractional Laplacian at $x  = x_i$.

\subsection{Approximation of the singular part of the integral}
In this  subsection, we deal with the singular integral \eqref{IsingDefn},
whose approximation is shown to be a rescaled second derivative at $x$.  

First, assuming that  $\meas$ is even, and satisfies \eqref{eq:levymeas}, and that $u$ has a bounded second derivative,  we can symmetrize~\eqref{IsingDefn} to obtain,
\bq\label{si2}
I^\meas_{S}[u](x) = \int_0^{h}
  \big(2u(x)-u(x+ y) - u(x-y)\big) \meas(y) \dm y.
\eq
which is no longer  principal value integral.

Assuming that $u \in C^4$, substitute the Taylor expansion with exact remainder
\[
u(x \pm y) = u(x) \pm u'(x) y + \frac {y^2}{2} u''(x) \pm
\frac{y^3}{6}u'''(x) +   \frac{y^4}{24} u''''\big(\xi(y)\big)
\]
into~\eqref{si2}. 
The odd terms cancel, giving 
\begin{align}
I^\meas_{S}[u](x) &=  -{u''(x)}\int_0^{h}\
 y^2 \meas(y)\dm y -  \frac{1}{12}\int_0^{h}u^{(4)}\big(\xi(y)\big)  y^4  \meas(y) \dm y \cr
&=  -{u''(x)} \int_0^{h} y^2 \meas(y) \dm y - \frac{u^{(4)}(\bar{\xi})}{12}  \int_0^{h} y^4  \meas(y) \dm y,
\end{align}
for some  $\bar{\xi}$ in the interval $(x-h,x+h)$. 
When $\meas(y) = \meas^\al(y)$ we obtain
\begin{equation}\label{eq:FLsi}
I^{\nu^\al}_{S}[u](x) = -\Cal \left ( 
\frac{h^{2-\alpha}}{2-\alpha} u''(x) +
    \frac{u''''(\bar{\xi})}{12}\frac{h^{4-\alpha}}{4-\alpha}
    \right).
\end{equation}
Therefore, the approximation above replaces the singular integral~\eqref{IsingDefn} by a second derivative of $u(x)$. For a fully discrete approximation, we next replace 
 $u''(x)$ by central differences,
\[
u''(x) = \frac{u(x+h) - 2u(x) + u(x-h) }{h^2} +  \frac{u''''(\tilde{\xi})}{12} h^2
\]     
to obtain
\bq\label{SingIntApprox1}
I^{\nu^\al}_{S}[u](x) = -\Cal \frac{u(x+h) - 2u(x) + u(x-h) }{(2-\alpha)h^{\alpha}} 
- 
 \frac{M_4}{12}\frac{h^{4-\alpha}}{2-\alpha},
\eq
where $M_4$ is the absolute value of a linear combination of  fourth derivatives  of $u$.
This last equation leads to the fully discrete approximation of~\eqref{IsingDefn} at $x_i$, 
\bq\label{SingIntWeights}
I^{\nu^\al}_{S}[u](x_i) \approx -\Cal \frac{u_{i+1} - 2u_i + u_{i-1} }{(2-\alpha)h^{\alpha}}
=\frac{\Cal h^{-\alpha}}{2-\alpha}(u_i-u_{i+1})+
\frac{\Cal h^{-\alpha}}{2-\alpha}(u_i-u_{i-1}).
\eq

\subsection{The tail of the integral} 
\label{sec:interp}
Next we focus on the tail of the integral~\eqref{Tail}.  A natural idea is to use a simple quadrature rule to approximate it, for example the trapezoidal rule
for the integrand $\big( u(x)-u(x-y)\big) \nu(y)$.    However, the errors from the trapezoidal rule, which depend on a derivative of the integrand,  blow up as $h \to 0$.   Even for finite $h$, these errors are  too large to be used in practice. Instead, we take some extra effort  to integrate \emph{exactly} the weight function $\meas(y)$ multiplied by a polynomial interpolant of $v(y) := u(x_i)-u(x_i-y)$.  

In the tail region $|y|\geq h$,  we approximate the regular part of the integrand $v(y) = u(x_i)-u(x_i-y)$ by the interpolant,
\begin{equation}\label{eq:numinterp}
    \mathcal{P}v(y) = \sum_{j \in \mathbb{Z}} v(x_j) P_{j}(y-x_j)
=\sum_{j \in \mathbb{Z}} (u_i-u_{i-j}) P_{j}(y-x_j), 
\end{equation}
for some basis functions, $P_j$, defined so that $P_j(0) = 1$ and $P_j(x_k)=0$ for 
$k\neq 0$. The basis functions $P_j$ are  Lagrange basis polynomials extended to finite overlapping domains.  Specific examples, such as the first order tent function, are defined in subsection~\ref{sec:tent}. 

Substituting the interpolant~\eqref{eq:numinterp} into~\eqref{Tail}, we obtain 
\begin{equation}\label{eq:approxint}
  I^\meas_{T}[u](x_i) \approx 
\int_{|y|\geq h}     \mathcal{P}v(y)\meas(y)dy
=
   \sum_{j \neq 0}  (u_i-u_{i-j})
    \int_{|y|\geq h}
    {P_j(y-x_j)} \meas(y) dy
\end{equation}
In section~\ref{sect:explicitweights} we evaluate the last integral above \emph{exactly}, since both the weight function $\nu(y)$ and the basis polynomials $P_j$ are known, avoiding the additional error from a direct numerical quadrature involving $\meas(y)$.

The approximation~\eqref{eq:approxint} defines, for fixed $x_i$,  the weight 
$\tilde{w}_{i,j}  = \int_{|y|\geq h} P_j(y-x_j) \meas(y) dy$.
In principle, we could use a different interpolation at each $x_i$.  However, as already indicated by the notation, 
the basis function  $P_j(y)$ is assumed independent of $i$, which means the weight $\tilde{w}_{i,j}$ is also independent of $i$. 
Since we also assume $P_j(y)$ is symmetric, this leads to a choice of weights in~\eqref{FLh}, given by 
\begin{equation}\label{eq:weightsform}
w_j = w_{-j}
=\int_{|y|\geq h} P_j(y-x_j) \meas(y) \dm y.
\end{equation}

\begin{remark}  We must be careful evaluating~\eqref{eq:weightsform} on the boundary of intervals, both near the singular interval ($|y|\approx h$ or $|j|=1$), and near $|y| \approx L$ when the computational domain is truncated to  $[-L,L]$, in order to avoid double counting of the weights.
\end{remark}

\begin{remark} We have to isolate $I^\meas_S[u]$ from~\eqref{eq:SI}, because of the singularity of $\nu(y)$ at the origin.
If~\eqref{eq:weightsform} is integrated on $\mathbb{R}$ instead of
the interval $(-\infty,-h)\cup (h,\infty)$, 
then $w_{\pm 1}$ may diverge for most functions $P_{\pm 1}$ with 
$\alpha \in (1,2)$. The splitting of $I^\meas_S[u]$ and 
$I^\meas_T[u]$ avoids this problem at the singularity.
\end{remark}

Using the formula~\eqref{eq:weightsform}, the weights are obtained 
when $P_j(y)$ is replaced by piecewise polynomials.  
These functions are local Lagrange basis polynomials on each interval.
 The first order basis function are piecewise linear ``tent'' functions;  the second order polynomials comprise two families of piecewise quadratic functions (see Figure~\ref{fig:Tents}).

Before giving the explicit weights, we discuss the errors in the approximation.

\subsection{Error from the quadrature in the tail and the overall error} 

In this section we estimate the errors from the approximation 
of $I^{\nu^\alpha}_T[u]$ using~\eqref{eq:approxint}, with the weights defined by~\eqref{eq:weightsform}. First we have the following lemma related to the error in quadrature using Lagrange interpolation.

\begin{lem}\label{lem:poly}
Let $\mathcal{P}^k f$ be the Lagrange polynomial interpolant with order $k$ of the function $f$ on the interval $[a,b]$, using equally spaced nodes which include the endpoints.
  Then for any nonnegative measure $\meas(t)$,
 \begin{equation}\label{pfQ1}
    \Big| \int_{a}^{b} \left(\mathcal{P}^k f(t) - f(t)\right) \meas(t)\dm t\Big| 
\le (b-a)^{k+1}  \frac{M_{k+1} }{2^{k+1} (k+1)! }  \int_{a}^{b} \meas(t)
\dm t,
\end{equation}
where $M_{k+1}$ is a bound on the $(k+1)$-th order derivative of $f$ on $[a,b]$.

\end{lem}

\begin{proof}  The proof is elementary, but we include it for completeness.
The Lagrange interpolant on $[a,b]$ satisfies
\[
\mathcal{P}^k f(t) - f(t) = 
\frac{f^{(k+1)}\big( \xi(t)\big)}{(k+1)!} (t - t_0)\cdots (t - t_k),
\]
for some $\xi(t)$ in the interval $(a,b)$, 
where $a =t_0<t_1<\dots<t_k = b$ are the interpolation nodes.
So
\begin{align*}
\Big|\int_{a}^{b} \left(\mathcal{P}^k f(t) - f(t)\right) \meas(t) \dm t \Big|
&= \Big|\int_{a}^{b}  \frac{f^{(k+1)}\big( \xi(t)\big)}{(k+1)!} (t - t_0)
\cdots (t - t_k)  \meas(t) \dm t\Big| \cr 
&\leq \frac{M_{k+1}}{(k+1)!}\int_a^b \big| (t-t_0)\cdots (t-t_k)\big|
\meas(t)dt
\end{align*}
For any $t\in (a,b)$, we have the bound 
$|(t-t_0)\cdots (t-t_k)|\leq (b-a)^{k+1}/2^{k+1}$.  
Combine these estimates to obtain~\eqref{pfQ1}.
\end{proof}

The Composite Exact Quadrature Rule is given by piecewise polynomial interpolation of degree $k$ on the each sub-interval of size $kh$ in the interval $(h, \infty)$ or $(-\infty,-h)$. In other words, the piecewise
 interpolations $\mathcal{P}^kv$ in~\eqref{eq:approxint} are polynomials
of degree $k$ on the sub-intervals $[h,(k+1)h)],[(k+1)h,(2k+1)h],\cdots$.
The error in this approximation  is summarized in the following lemma.

\begin{lem}
The error for the Composite Exact Quadrature Rule is given by 
\begin{equation}\label{WeightInTail}
    \Abs{  \int_{\abs{t} \ge h} \left( \mathcal{P}^k f(t) - f(t)  \right ) \meas(t) \,dt }
\le 
\frac{  k^{k+1} M_{k+1} }{  {2^{k+1}(k+1)!}} 
\,
h^{k+1}  \int_{h}^\infty  \big[\meas(t) + \meas(-t)\big]   \dm  t
\end{equation}
where $M_{k+1}$ is a bound on the $(k+1)$-th order derivative of $f$.
\end{lem}

\begin{proof}
Write the integral as a sum over sub-intervals of length $kh$ and apply Lemma~\ref{lem:poly} on each sub-interval.  Then the combined error is bounded by 
\[
(kh)^{k+1} \frac{M_{k+1}}{2^{k+1}(k+1)!} 
\left ( 
 \int_{h}^\infty  \meas(t)    \dm  t
 + 
  \int_{-\infty}^{-h}  \meas(t)    \dm  t
\right )
\]
which gives the result~\eqref{WeightInTail}.
 \end{proof}


In particular, from the conditions on the measure~\eqref{eq:levymeas} and the additional assumption $\nu(y) \sim |y|^{-1-\alpha}$ near the origin, 
\[
\int_h^\infty \big[\meas(y) + \meas(-y)\big] \dm y = 
\int_{\abs{y} \ge 1} \meas(y) \dm y+ 
\int_{h \le \abs{y} \le 1}  \meas(y) \dm y \sim h^{-\alpha}.
 \]
So  the error  in approximating the tail of the integral~\eqref{Tail} is bounded by
\bq\label{ErrorTail}
\bO(h^{k+1-\al}).
\eq
This, together with the error $O(h^{4-\alpha})$ from~\eqref{SingIntApprox1}, implies the following overall error
of the method.
\begin{lem}\label{lem:accuracy1}
Suppose $u\in C^{4}$.  Then the combined error for  the finite difference scheme~\eqref{FLh} is $O(h^{2-\alpha})$
using weights obtained from exact quadrature using linear interpolation and $O(h^{3-\alpha})$ 
using weights obtained from exact quadrature using  quadratic interpolation.
\end{lem} 

%

\begin{remark}
In fact, a careful examination of the errors from the singular part
and the tail above suggests at least two improvements. 
Higher order interpolation polynomials can be used, but 
the corresponding weights computed from~\eqref{eq:weightsform}
can be negative and many desired properties like maximum principle
are lost. Another possible improvement is to choose a wider
interval for the singular part of the integral $I_S^{\nu}[u]$, say
from $h$ to $h_0$. The leading order errors now become 
$h_0^{4-\alpha}$ (from $I_S^{\nu^\alpha}[u]$) and $h^{k+1}h_0^{-\alpha}$ (from $I_T^{\nu^\alpha}[u]$). 
The balance of these two errors leads to the optimal 
choice $h_0\sim h^{(k+1)/4}$ and consequently an
overall error of $O(h^{(k+1)(4-\alpha)/4})$. Since the 
improvement of the order for linear or quadratic interpolation is just
a fractional power of $\alpha$, we do not pursue this direction here.
\end{remark}

\section{Explicit calculation of the weights}
\label{sect:explicitweights}
\subsection{Piecewise Linear and quadratic interpolants}  
\label{sec:tent}
Piecewise polynomial interpolation is a standard topic in elementary numerical analysis.   For our purpose, we recall the formulas for the first and second order interpolants.  These formulas will be used  to derive explicit weights $w_j$ from the semi-exact quadrature rules against the weight function $\meas(y)$.  

Given the values $\{ v(x_j) \}$ on the grid $x_j=jh$ with spacing $h$, we defined the following piecewise linear and quadratic interpolants (See Figure~\ref{fig:Tents}).

The piecewise linear interpolant is 
given by 
\[
    \mathcal{P}^1_h v(x) = \sum_{j\in \mathbb{Z}} v(x_{j}) \hhat(x-x_{j}),
\]
where $\hhat(t)$ is the ``tent'' function
\[
\hhat(t) = 
\begin{cases}
    1 - \abs{t}/h \qquad & \abs{t} \le h,
\\
0 & \text{ otherwise}.
\end{cases}
\]
The  piecewise quadratic interpolant is given by 
\[
    \mathcal{P}^2_hv(x) = \sum_{j \text{ even}} v(x_{j}) \Qe(x-x_{j})  
    + \sum_{j \text{ odd}}
    v(x_{j}) \Qo(x-x_{j}),
\]
where $\Qe(x)$ is the quadratic Lagrange polynomial which 
interpolates $(0,1,0)$ at $(-h,0,h)$:
\[
\Qe(t) = 
\begin{cases} 
    1-t^2/h^2 \qquad & \abs{t} \le h,
\\
\quad 0 & \text{ otherwise},
\end{cases}
\]
and $\Qo(x)$ is  a piecewise quadratic Lagrange polynomial which 
interpolates $(0,0,1)$ on the left and $(1,0,0)$ on the right:  
\[
\Qo(t) = 
\begin{cases} 
    1-3\abs{t}/2h+t^2/2h^2    
\quad & \abs{x} \le 2h,
\\
\quad 0  & \text{ otherwise}.
\end{cases}
\]

\begin{figure}
 \begin{center}
  \includegraphics[totalheight=0.24\textheight]{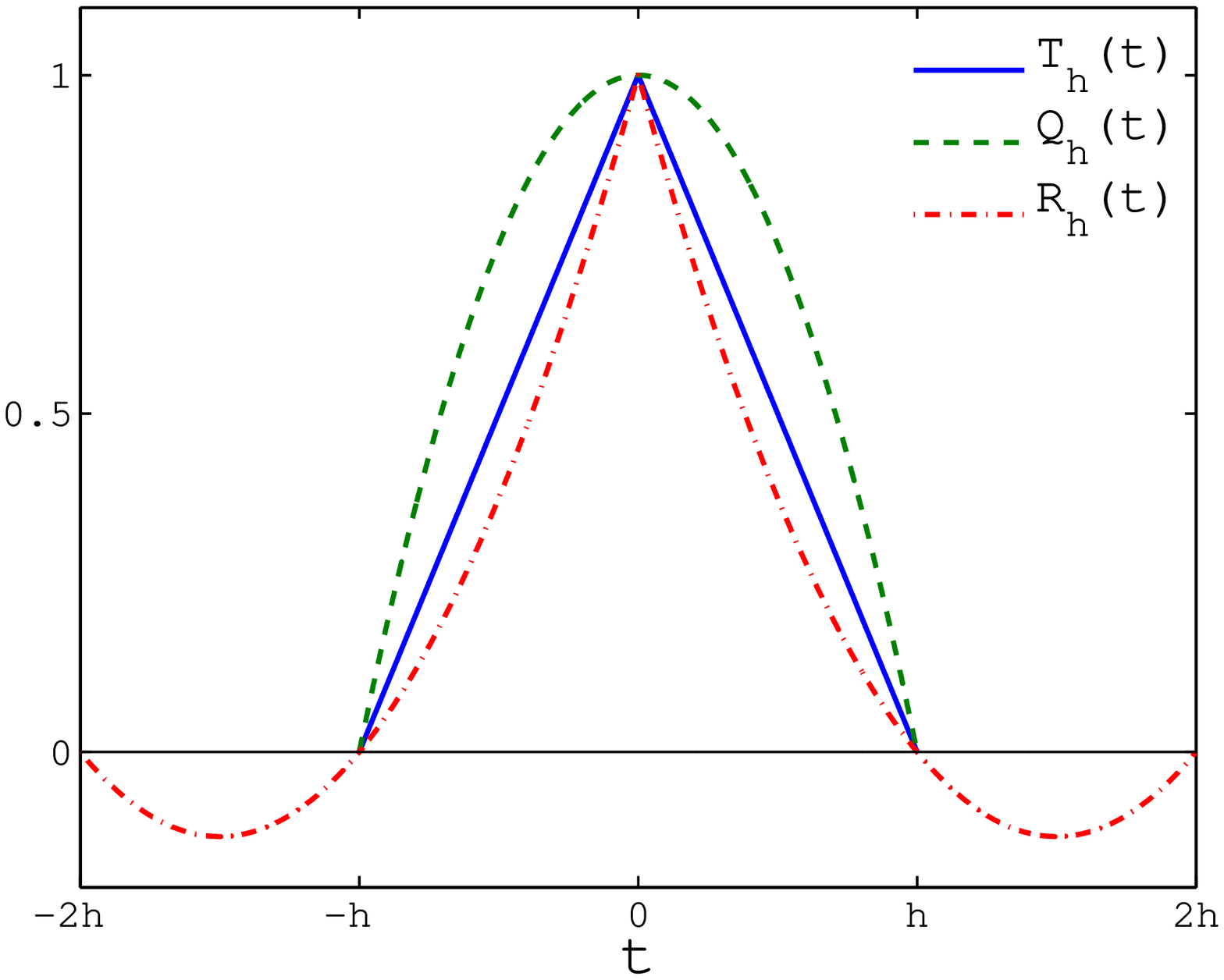}
  \includegraphics[totalheight=0.24\textheight]{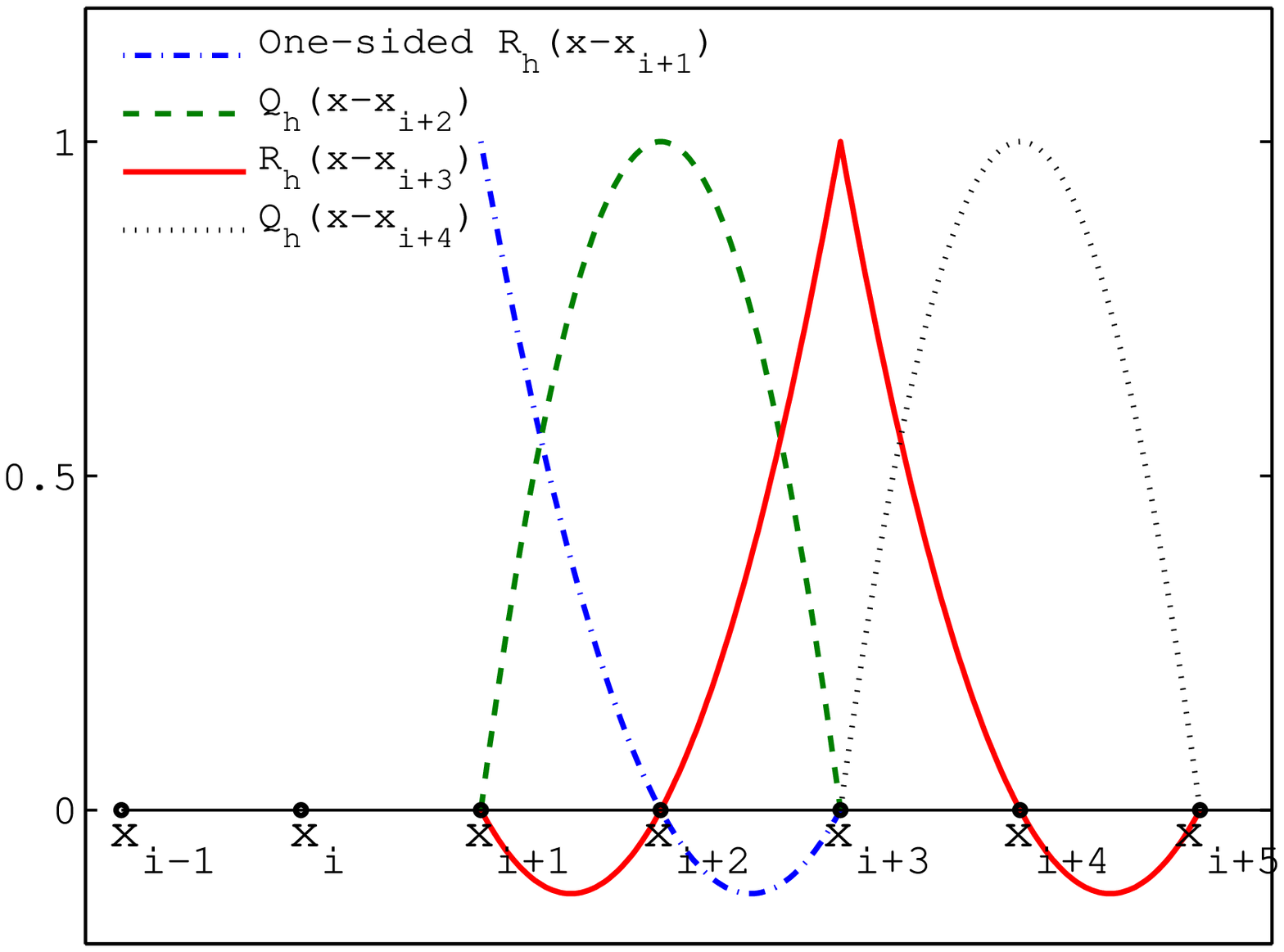}
 \end{center}
\caption{Left: the tent function $\hhat$, and the quadratic interpolants $\Qe, \Qo$. Right: the basis functions for quadratic interpolation.}
\label{fig:Tents}
\end{figure}

\subsection{Formulas for weights in semi-exact quadrature} 
To derive the explicit expressions for the weights,
we need the following lemma,  which is proved using integration by parts.

\newcommand{\Lhh}{\Big |_{-h}^h}

\begin{lem}
\label{lem:tent_int_by_parts}
Let $F(t)$ be a $C^2$ function and $G(t)$ be a $C^3$ function 
on $[-2h,2h]$, 
$T_h(t), Q_h(t), R_h(t)$ be defined above.  Then
\begin{align}
\label{TentIBP}
\int_{-h}^h \hhat(t) F''(t) \dm t &=
\frac{1}{h}\big( F(h)-2F(0)+F(-h)\big),
\\ \label{Qinbyparts}
\int_{-h}^h Q_h(t)G'''(t)dt &= \frac{2}{h}\big(G'(h)
+G'(-h)\big)-\frac{2}{h^2}\big(G(h)-G(-h)\big),
\\ \label{Rinbyparts} 
\int_{-2h}^{2h} R_h(t)G'''(t)dt &=
-\frac{G'(2h)+6G'(0)+G'(-2h)}{2h}+\frac{G(2h)-G(-2h)}{h^2}.
\end{align}
Also, we will need the following one-sided integrals at the boundaries
\begin{align}\label{boundint}
\int_{0}^h \hhat(t) F''(t) \dm t &=
-F'(0) + \frac {F(h)-F(0)}{ h}, \\
\label{boundint2}
\int_{0}^{2h} \Qo(t) G'''(t) \dm t &=
-G''(0) - \frac{G'(2h)+3G'(0)}{2h}
+\frac{G(2h)-G(0)}{h^2}.
\end{align}
\end{lem}

We now combine the linear or quadratic interpolation with Lemma~\ref{lem:tent_int_by_parts} to obtain the weights.
First, rewrite~\eqref{eq:weightsform} in a 
form consistent with the above lemma,
\[
    w_j = \int_{|y|\geq h} P_j(y-x_j)\nu(y)dy =
\int_{|x_j+t|\geq h} P_j(t) \nu(t+x_j)dt,
\]
where $P_j(t)$ is $T_h(t)$, $Q_h(t)$ or $R_h(t)$ defined above.

\begin{defn}\label{defn:fg}
Define the functions $F(t)$ and $G(t)$ to be primitives of the weight function: $F''(t) = G'''(t)=\nu(t)$. In particular, 
in the case of fractional Laplacian, $\nu^\alpha(t)=C_{1,\alpha}t^{-1-\al}$, 
\[
F(t) = 
\begin{cases}
\frac{C_{1,\alpha}}{(\al-1)\al} \abs{t}^{1-\al},& \al \not = 1
\\
- C_{1,\alpha}\log |t|, & \al = 1,
\end{cases}
\qquad
G(t) = 
\begin{cases}
\frac{C_{1,\alpha}}{(2-\al)(\al-1)\al} \abs{t}^{2-\al},& \al \not = 1
\\
C_{1,\alpha}(t-t\log |t|),  & \al = 1.
\end{cases}
\]
\end{defn}

For the piecewise linear interpolants, using~\eqref{eq:weightsform}, 
 for $|j|>1$ the weights  are given by
\begin{align*}
 w_j^T=\int_{|x_j+t|\geq h} T_h(t)F''(x_j+t)dt
 &=\frac{1}{h}\big( F(x_j+h)-2F(x_j)+F(x_j-h)\big) \cr
 &=h^{-\alpha}\big( F(j+1)-2F(j)+F(j-1)\big).
\end{align*}

For the piecewise quadratic interpolants, for $|j|>1$ the weights  are given by
\begin{align*}
 w_j^Q &= \int_{|x_j+t|\geq h} Q_h(t)G'''(x_j+t)dt \cr
 &=  2h^{-\alpha}\big[  G'(j+1)+G'(j-1)
 - G(j+1)+G(j-1)\big],  & \text{$j$ even}.
\\
 w_j^Q &= \int_{|x_j+t|\geq h} R_h(t)G'''(x_j+t)dt \cr
 &= h^{-\alpha}\left[
 -\frac{G'(j+2)+6G'(j)+G'(j-2)}{2}
 +G(j+2)-G(j-2)
 \right],  & \text{$j$ odd}.
\end{align*}

When $j=1$, one part of the domain of the integral falls into the singular interval $[-h,h]$ that is isolated in $I_S^\nu[u]$, so 
we need to use the one-sided integrals~\eqref{boundint} or~\eqref{boundint2} for the other part of the integal.
Combined with the contribution on the singular interval in~\eqref{SingIntWeights}, the weights for $j=1$ become
\begin{equation}
    w_1^T = \frac{C_{1,\alpha}h^{-\alpha}}{2-\alpha} + 
    \int_0^h T_h(t)F''(x_1+t)dt
    = h^{-\alpha}\left[\frac{C_{1,\alpha}}{2-\alpha}  
-F'(1)+F(2)-F(1)                \right]
\end{equation}
and 
\begin{equation}
    w_1^Q = h^{-\alpha}\left[
        \frac{C_{1,\alpha}}{2-\alpha}-G''(1)-\frac{G'(3)+3G'(1)}{2}+G(3)-G(1)
        \right].
\end{equation}
The weights are even, $w_{-j}^T=w_j^T$ and $w_{-j}^Q =
w_j^Q$, since we chose the interpolation functions to be symmetric.

We summarize the expression of the weights as follows.
\begin{defn}\label{defn:weights}
The weights corresponding to the finite difference
method~\eqref{FLh} of the fractional Laplacian for  piecewise linear interpolation are given by 
\[
 w_j^T=
 h^{-\alpha}
 \begin{cases}
\dfrac{C_{1,\alpha}}{2-\alpha}  
-F'(1)+F(2)-F(1), & j = \pm 1,
\\
F(j+1)-2F(j)+F(j-1), & j = \pm 2, \pm 3, \cdots.
\end{cases}
\]
The weights corresponding to piecewise quadratic interpolation are
\[
w_j^Q =
h^{-\alpha}
 \begin{cases}
\dfrac{C_{1,\alpha}}{2-\alpha}-G''(1)-\dfrac{G'(3)+3G'(1)}{2}+G(3)-G(1), & j = \pm 1,
\\
 2\big[  G'(j+1)+G'(j-1)
 - G(j+1)+G(j-1)\big], & j = \pm 2, \pm 4, \pm 6, \cdots
\\
 -\dfrac{G'(j+2)+6G'(j)+G'(j-2)}{2}
 +G(j+2)-G(j-2), & j = \pm 3,\pm 5,\pm 7,\cdots.
\end{cases}
\]
The weight $w_0$ can be arbitrary, because it does not enter into \eqref{FLh}.
\end{defn}

\subsection{Discussion of properties of the weights}\label{sec:properties weights}
We collect some important properties of the weights here, focusing on the  case $\meas = \meas^\alpha$ of the fractional Laplacian.  The more general cases with positive L\'{e}vy measures can be studied in a similar way.

\emph{Positivity of the weights.} 
The weights obtained from linear and  quadratic interpolation, summarized in Definition~\ref{defn:weights},
are positive.  The positivity  is easy to see for weights derived from linear interpolation, and for the even weights derived from quadratic interpolation, since $T_h$ and $Q_h$ are positive. 
Although it is not immediately clear for odd weights from quadratic interpolation, their positivity has been  verified 
numerically. 

The positivity of the weights $w_j$ for $j\neq 0$ is
important for the maximum principle, stability (or monotonicity) and convergence of the discrete solutions for the extended Dirichlet problem. 
Notice that the appearance of negative weights is exactly the reason we do not use higher order polynomial interpolation.  For instance  using cubic interpolation polynomials leads to $w_3 < 0$ with, and similar negative weights arise  for higher
order polynomials. In the extreme case of infinitely many interpolation nodes, the global basis function $P_j(x) = \mbox{sinc} \frac{x}{h} = \frac{\sin \pi x/h}{\pi x/h}$ is used and the corresponding weights $w_j$ with even indices $j$ are in fact all negative when $\alpha > 1$, see below.

\emph{Scaling of the weights}
The weights from both linear and quadratic interpolation 
have certain scaling properties inherited from the measure $\nu^\alpha(y) \,dy$. 
The dependence on $h$ is always $h^{-\alpha}$, reflecting the fact that 
the fractional Laplacian has a fractional derivative of 
order $\alpha$. Furthermore, the weights $w_j$ also decay
at a rate $j^{-1-\alpha}$, once again because of the fact
$\nu^\alpha(y) \sim |y|^{-1-\alpha}$ as $|y|\to\infty$. The scaling
rate can also be checked for both weights using the asymptotic expansions
of the auxiliary functions $F$ and $G$ (and their derivatives) in $j$, when $j$ is large.

\emph{Consistency when $\al \to 2^-$.}
In the limit $\al \to 2^-$, we recover 
the standard three point central difference scheme for $-\partial_{xx}$.
In fact, using the explicit expression $\Cal$ from~\eqref{eq:constC} and 
the weights $w_j$ either from linear or quadratic interpolation, we can get 
\[
 \lim_{\al \to 2^-}  w_j = \begin{cases}
                                           h^{-2}, \qquad & j = \pm 1,\cr
0, &\mbox{otherwise},
                                          \end{cases}
\]
and hence for both $w^T$ and $w^Q$,
\[
 \lim_{\al\to 2^{-}} [(-\Delta_h)^{\alpha/2}u]_i = \frac{u_{i}-u_{i+1}}{h^2}+\frac{u_{i}-u_{i-1}}{h^2}
=-\frac{u_{i+1}-2u_i+u_{i-1}}{h^2}.
\]

\emph{CFL condition.} 
Using an explicit method for the parabolic evolution operator leads to a restriction on the time step, $dt$.  The bound on the time step is given by 
\[
dt \le \sum_{j \not= 0} w_j,
\]
the sum of the weights, which is derived here.

We expect an upper bound of the form $Ch^{-\alpha}$ for 
the sum of the total weights (except the irrelevant $w_0$), 
based on the fact that  $w_j =
O(|j|^{-1-\alpha}h^{-\alpha})$ and the expression $\Cal $
given by~\eqref{eq:constC}. 
In fact, this sum can be worked out explicitly,
 \begin{equation}\label{eq:sumT}
   \sum_{j\neq 0} w_j^T = 2h^{-\alpha}\Big[\frac{C_{1,\alpha}}{2-\alpha}
-F'(1)\Big] = \frac{2^{\alpha}\Gamma\big( (\alpha+1)/2\big)}
{\pi^{1/2}\Gamma(2-\alpha/2)}h^{-\alpha}, 
 \end{equation}
and 
\begin{equation} \label{eq:sumQ}
   \sum_{j\neq 0} w_j^Q = 2h^{-\alpha}\Big[\frac{C_{1,\alpha}}{2-\alpha}
-G''(1)\Big]=\frac{2^{\alpha}\Gamma\big( (\alpha+1)/2\big)}
{\pi^{1/2}\Gamma(2-\alpha/2)}h^{-\alpha},
\end{equation}
the same for both weights. Another important observation is that, without the factor $h^{-\alpha}$,
the sums are uniformly bounded for any $\alpha \in (0,2)$, and do not 
degenerate as $\alpha$ approaches either 0 or~2. 

\section{Convergence of solutions to the Dirichlet boundary value problem}
\label{sec:proof}
In this section, we provide a convergence proof of solutions to the extended Dirichlet boundary value problem~\eqref{FLD}.
This gives a rigorous foundation for the scheme, by showing that for smooth enough data, the error of the solution (in the maximum norm) has the same order of accuracy as the local truncation error of the scheme. 

Some of the techniques we use to prove the convergence are similar to those for linear elliptic partial differential equations~\cite[Chapter4]{larsson2009partial}. We also have to find a \emph{super-solution} $v(x)$ to control the error on the interval $[-1,1]$,
where $v$ satisfies
\[
(-\Delta)^{\alpha/2}v(x) \geq  1,\quad x \in (-1,1)
\]
and $v(x)\geq 0$ on $(-\infty,-1]\cup [1,\infty)$.

The natural candidate for the super-solution is  $v_G(x) = K_\alpha (1-|x|^2)^{\alpha/2}_+ $ discovered by Getoor~\cite{MR0137148} with $K_\alpha = 2^\alpha\Gamma\big(1+\alpha/2\big)\Gamma\big((1+\alpha)/2\big)\big/{\sqrt{\pi}}$.
Here $v_G$ satisfies
\[
(-\Delta)^{\alpha/2}v_G(x) = 1, \quad x \in (-1,1),
\]
and $v_G(x)\equiv 0$ for $|x|\geq 1$.
This function is the probability density function for the expected exit time for a particle undergoing a random walk governed by the $\alpha$-stable L\'{e}vy process.
When $\alpha=2$, $v_G(x)=\frac{1}{2}(1-|x|^2)_+$  is the classical super-solution used to prove the convergence of finite difference methods to elliptic equations~\cite{larsson2009partial}. 
Therefore, it is natural to use the sampled function $v_i = v_G(ih)$ as a discrete super-solution of the problem
$(-\Delta_h)^{\alpha/2}v_i \geq 1$ if $|ih|<1$. However, a rigorous proof of this observation seems complicated, although numerically 
it is tested to be true for both $w_j^T$ and $w_j^Q$. Instead, we sample from a simple, discontinuous quadratic function.

\subsection{The discrete super-solution}
We first rewrite the expression for the operator $(-\Delta_h)^{\alpha/2}$ in~\eqref{FLh} as follows,
\begin{equation}\label{DelW}
\FLh u_i 
= \Cal \sum_{j=1}^\infty (-\delta_j u_i)w_j,
\end{equation}
where we use the notation $\delta_j u_i = u_{i+j} - 2u_i + u_{i-j}$.  Here with a temporary abuse of notation,  we removed the constant $\Cal$ from the weights, in order to show that the result is independent of $\al$. 

After discretization on a grid, the extended Dirichlet problem~\eqref{FLD} becomes
\begin{equation}\label{FLhD}\tag{$\text{FLD}_h$}
\begin{aligned}
\FLh u_i &= f_i, && \text{ for } i \in  \Ddh \\
u_i &= g_i,  && \text{ for  } i \in \DdCh
\end{aligned}
\end{equation}
where $\DdCh = \mathbb{Z}_h\setminus\Ddh$.
To be specific, we set 
\bq\label{DdInterval}
\Ddh = \{i \in \mathbb{Z}: \abs{i h} < 1\},
\eq  
although other domains could be treated.  Then we have the following
$\alpha$-independent discrete super-solution.

\begin{lem}\label{Lem41}
Define
\[
v(x) = 
\begin{cases}
4- x^2, & \abs{x} < 1\\
0, & \text{ otherwise}.
\end{cases}
\]
Then the grid function $v_i=v(ih)$ satisfies
\begin{equation}\label{vbarrier}
\begin{aligned}
\FLh v_i &\ge  1, && \text{ for } i \in  \Ddh,
\\
v_i &= 0,  && \text{ for  } i \in \DdCh,
\end{aligned}
\end{equation}
for $h$ small enough.
\end{lem}

Note that the function $v(x)$ defined above is discontinuous.

\begin{remark}\label{remarkBarrier} 
Here we explain why it is okay to use a discontinuous function for the super solution.
Perron's method is used to prove existence of solutions which attain the boundary values.  For Perron's method, continuous super-solutions, or barrier functions,  such as the one provided by Getoor, are needed.   However, for the stability results proved here, continuity is not required, 
 and it is more convenient to use a  discontinuous function. This simplicity is compensated by a larger (and thus cruder) constant in the estimates below. 
\end{remark}

\begin{proof}
We prove the lemma in two steps.

Step 1.  We will establish that 
\begin{equation}\label{eq:step1}
-\delta_j v_i \ge \min(2, 2(jh)^2), \quad \text{ for } i \in \Ddh,\ j>0.
\end{equation}

Case (i).  Suppose $i+j$ and $i-j$ are both in  $\Ddh$.   Then 
\[
-\delta_j v_i = h^2
\left (   
(i + j)^2 + (i-j)^2 - 2i^2
\right )
= 2(jh)^2.
\]

 Case (ii).  Suppose $i+j$ and $i-j$ are both outside $\Ddh$.  Then
\[
-\delta_j v_i = 2v_i  = 8 -  2(ih)^2  \ge 6.
\]

 Case (iii).  Suppose exactly one of $i\pm j$ is in $\Ddh$.  
Set $y = (i \pm j)h$, choosing the sign so that $\abs{y} < 1$.  Then
\[
-\delta_j v_i = -v(y) + 2u(x) = 4 - 2x^2 + y^2 \ge 4 - 2x^2 \ge 2.
\]
So \eqref{eq:step1} holds in each case.

Step 2. 
Using \eqref{eq:step1}, we have for $i \in \Ddh$,
\[
\FLh v_i \ge  \sum_{j=1}^\infty w_j \min(2,2(jh)^2) 
=   2\Big(\sum_{\abs{jh} \geq 1}  w_j +
\sum_{\abs{jh}< 1} w_j (jh)^2\Big) .
\]
The two sums above can be estimated separately by integrals.  
For the first sum,
\begin{equation}\label{eq:1stSPsum}
\sum_{|jh|\geq 1} 
w_j \approx 2C_{1,\alpha}
\int_1^\infty y^{-1-\al} dy = \frac{2C_{1,\alpha}}{\al},
\end{equation}
with an error $O(h^{k+1})$ depending on the degree $k$ of the polynomial used to derive $w_j$. For the second sum, 
\begin{equation}\label{eq:2ndSPsum}
\sum_{|jh|< 1} 
 w_j \approx 2C_{1,\alpha} 
\int_0^1 y^{1-\al} dy = \frac{2C_{1,\alpha}}{2-\alpha},
\end{equation}
with an error $O(h^{k+1-\alpha}+h^{4-\alpha})$. 
When $h$ is small enough,
\[
\FLh v_i \ge \Cal \left( \frac{4}{2-\al} + \frac{4}{\al} \right)
= \frac{2^{\alpha+1}\Gamma\big( (\alpha+1)/2\big)}
{\Gamma(1/2)\Gamma(2-\alpha/2)} > 1,
\]
since $2^\alpha>1$, $\Gamma\big( (\alpha+1)/2\big) > 
\pi^{1/2}=\Gamma(1/2)$ and $\Gamma(2-\alpha/2)<\Gamma(2)=1$. 
\end{proof}

\subsection{Maximum principle and convergence proof}
To prove the convergence of the discrete problem~\eqref{FLhD}, we begin with a  standard result: positivity of the weights in a discrete Laplace operator implies a maximum principle. 
  The proof is included for completeness, similar results can be found in~\cite[Chapter 6]{larsson2009partial}.

\begin{lem}[Maximum principle for~\eqref{FLhD}] \label{lem:MP}
Let $u$ satisfy $\FLh u_i \le 0$ for $i \in \Ddh$.  Then
\[
\max_{i \in \Ddh} u_i \le \max_{i \in \DdCh} u_i.
\]
Similarly, if $\FLh u \ge 0$ for $i \in \Ddh$, then 
\[
\min_{i \in \Ddh} u_i \ge  \min_{i \in \DdCh} u_i.
\]
\end{lem}
\begin{proof}
Suppose $\FLh u_i \le 0$ for $i \in \Ddh$. If
\[ 
u_{i_0} = \max_{i \in \Ddh} u_i > \max_{i \in \DdCh} u_i,
\] 
$u_{i_0}$ is the global maximum on $\mathbb{Z}_h$. As a result, 
\[
0\geq \FLh u_{i_0}=\sum \big(u_{i_0}-u_{i_0-j}\big) w_j > 0,
\]
which is a contradiction. So $u$ cannot have a global maximum on $\Ddh$, and $\max_{i \in \Ddh} u_i \le \max_{i \in \DdCh} u_i$, as desired.
The second result follows by a similar argument.
\end{proof}

Define the notation for the maximum norm of a grid function on the set of indices $I$:
\[
\norm{u}_{I} = \max_{i \in I} \abs{u_i}.
\]
\begin{lem}
For any grid function $u: \mathbb{Z}_h \to \R$ and $\Ddh = \{ i: |ih|<1\}$,
we have
\begin{equation}\label{normCrhs}
\norm{u}_{\Ddh} \le \norm{u}_{\DdCh} + 4\norm{\FLh u}_{\Ddh}.
\end{equation}
\end{lem}
\begin{proof}
Use the function $v$ from Lemma~\ref{Lem41} and  set 
\[
z =  u -  \norm{\FLh u}_{\Ddh} v.
\]
Then for any $i \in D_h$,
\[
\FLh z_i =  \FLh u_i - \norm{\FLh u}_{\Ddh} \FLh v_i 
\le 0.
\]
By the maximum principle in Lemma~\ref{lem:MP}, $\max_{\Ddh} z \le \max_{\DdCh} z$.   Since $v = 0$ on $\DdCh$ and $|v|\leq 4$ on $\Ddh$, we have
\[
\max_{\DdCh} z = \max_{\DdCh}  u.
\]
and 
\[
\max_{\Ddh} u \le \max_{\DdCh}  z + \norm{\FLh u}_{\Ddh} 
\max_{\Ddh} v \le \norm{u}_{\DdCh} + 4\norm{\FLh u}_{\Ddh}.
\]
Similarly, we can show that $ \min_{\Ddh} u \geq  -\norm{u}_{\DdCh} - 4\norm{\FLh u}_{\Ddh}$ and this completes the proof of~\eqref{normCrhs}.
\end{proof} 

\begin{defn}[Solution Error and Local Truncation Error]
Let $u$ be the solution of \eqref{FLD} and let $u^h$ be the solution of \eqref{FLhD}.
Define the approximate solution error as
\[
e^h_j = u(x_j) - u^h_j,
\]
and the truncation error as
\[
r^h_j = \FL u(x_j)  - \FLh u_j.
\]
\end{defn}
Next, we combine the previous results to prove the convergence
of the Dirichlet problem~\eqref{FLhD}.

\begin{thm}\label{thm1}
Let $u$ be the solution of \eqref{FLD} and let $u^h$ be the solution of \eqref{FLhD}.
Assume that $u$ is smooth enough for the truncation error estimate to be valid.
The maximum of the approximate solution error is bounded by (a constant times) the maximum of the truncation error:
\begin{equation}\label{errbound}
\norm{e^h}_{\Ddh} \le  4 \norm{r^h}_{\Ddh}.
\end{equation}
\end{thm}
\begin{proof}
Using the definition of the error $e^h_j = u(x_j) - u^h_j$ and the residual we have
\begin{align*}
\FLh e^h_j &= \FLh u_j - \FLh u^h_j  = r^h_j.
\end{align*}
Then, using \eqref{normCrhs}, along with the fact that $e^h = 0$ on $\DdCh$, we obtain~\eqref{errbound}.
\end{proof}

In particular, the convergence rate is given by the global accuracy from Lemma~\ref{lem:accuracy1}, which is $O(h^{2-\alpha})$ for $w^T$ and 
$O(h^{3-\alpha})$ for $w^Q$.

\section{Far field boundary conditions} 
\label{sec:farfieldBC}
In this section, we discuss some practical issues related to the 
truncation of the computational domain and the implementation of the
boundary conditions.

\subsection{Truncation of the domain}
The discrete operator~\eqref{FLh}  involves an infinite sum, and this is the case 
even in the extended Dirichlet problem~\eqref{FLhD} which involves a finite number of unknowns.   In practice, the computational domain has to be truncated.  If the function $g(x)$ does not decay fast enough, the contribution of the integral~\eqref{eq:SI} outside this domain should be approximated.   As we see from the numerical example below, the inclusion of  the truncated boundary terms  is essential for an accurate evaluation of the fractional Laplacian.

We can simply truncate \eqref{FLh} at a finite value $M$  of the index $j$, resulting in the operator
\begin{equation}\label{eq:truncFL}
(-\Delta_h)_M^{\alpha/2} u_i = \sum_{j=-M}^M (u_i-u_{i-j})w_j.
\end{equation}
However, large errors can accumulate because of the slow decay of 
function $u_i$ and the weights $w_j \approx j^{-1-\alpha}$.  The singular integral expression~\eqref{eq:SI} allows us to take advantage of analytical formulas for the truncated terms, reducing the overall error
significantly. More precisely, when $\FL u(x_i)$ is written as
\begin{equation}\label{eq:bdexpand}
\underbrace{\int_{-L_W}^{L_W} \big(u(x_i)-u(x_i-y)\big)\nu(y) dy}_\textrm{(I)}
+ \underbrace{\int_{|y|>L_W}u(x_i)\nu(y)dy}_\textrm{(II)}
- \underbrace{\int_{|y|>L_W}u(x_i-y)\nu(y)dy}_\textrm{(III)},
\end{equation}
with $L_W = Mh$, the truncated sum~\eqref{eq:truncFL} approximates the first term (I), with an error, $O(h^{3-\alpha})$ or $O(h^{2-\alpha})$
for $w^T$ and $w^Q$, respectively, which  is independent of $M$. 

We next consider the second term 
$\int_{|y|>L_W}u(x_i)\nu(y)dy = u(x_i)\int_{|y|> L_W} \nu(y)dy$. 
For general $\nu(y)$,  this integral may require numerical quadrature, however in the case of primary interest here,  when $\nu(y)=\nu^\alpha(y)$,  we can integrate to  obtain
\begin{equation}\label{eq:bc2}
 \textrm{(II)} = 2u(x_i)C_{1,\alpha}\int_{L_W}^\infty y^{-1-\alpha}dy
=\frac{2C_{1,\alpha}}{\alpha (L_W)^\alpha}u(x_i).
\end{equation}
This prefactor $2C_{1,\alpha}(L^W)^{-\alpha}/\alpha$ can also be
approximated alternatively by the sum\footnote{When the domain is truncated at $L_W=Mh$, $w_{\pm M}$ are evaluated using the one-sided integrals~\eqref{boundint} or~\eqref{boundint2}, and they are
different in the two sums $\sum_{|j|\leq M} w_j$ and $\sum_{|j|\geq M} w_j$.
} $\sum_{|j|\geq M} w_j$ = $\sum_{j\neq 0} w_j
-\sum_{|j|\leq M} w_j$, with the total sum $\sum_{j\neq 0} w_j$ given  in~\eqref{eq:sumT} or~\eqref{eq:sumQ}.

The approximation of the last term (III) depends on the situation at hand.
For the extended Dirichlet problem~\eqref{FLD}, assuming the domain 
$D$ is the interval $[-L,L]$, then $u(x_i+y)=g(x_i+y)$ is given on $|y|>L_W$, if $L_W\geq 2L$ (notice that we only need $|x_i|\leq L$). Therefore, this 
term is known and its evaluation, which may require numerical quadrature, is independent of our scheme; in 
the special case $g\equiv0$, this term is identically zero.

\subsection{Approximations of the boundary term (III)}
\label{sec:bciii}
In addition to the extended Dirichlet problem, we wish to consider equations posed on the whole space $\R$.  This situation occurs in evolution equations like the fractional Burgers equation~\cite{MR1637513,MR2227237} or the fractional porous medium equations~\cite{MR2737788,MR2817383,MR2847534}.
In the case where the equation is posed on the whole real line, some knowledge asymptotic behavior of the function is required.
In the examples cited above, solutions usually develop an algebraic tail, with or without explicit exponents. 
In this case, it is impractical to choose  $L_W$ large enough so that the contribution of (III) in~\eqref{eq:bdexpand} can be ignored.
On the other hand, this algebraic tail can help us extract the leading order
contribution of (III).

If $u(x)$ decays to zero algebraically, 
$u(y) \sim |y|^{-\beta}$, then we can write 
\begin{align*}
u(y) &\approx u(L)L^\beta |y|^{-\beta},&&  \text{ as } y \to \infty,
\\
u(y) &\approx u(-L)L^\beta |y|^{-\beta}, && \text{ as } y \to -\infty. 
\end{align*}
Substituting these asymptotic expansions into (III), 
and setting  $\nu=\nu^\alpha$ and $L_W \geq 2L$, we obtain
\begin{subequations}
\begin{align}\label{eq:bc31}
\int_{L_W}^\infty u(x_i-y)\nu^\alpha(y)dy&\approx 
C_{1,\alpha}u(-L)L^\beta
\int_{L_W}^\infty (y-x_i)^{-\beta}y^{-1-\alpha} dy\cr
&=\frac{C_{1,\alpha}u(-L)L^\beta}{ (\al+\beta)(L_W)^{\alpha+\beta}}\ {}_2F_1\Big(\beta,\al+\beta;\al+\beta+1;\frac{x}{L_W}
\Big)
\end{align}
and
\begin{equation}\label{eq:bc32}
 \int^{-L_W}_{-\infty} u(x_i-y)\nu^\alpha(y)dy \approx 
\frac{C_{1,\alpha}u(L)L^\beta}{ (\al+\beta)(L_W)^{\alpha+\beta}}\ {}_2F_1\Big(\beta,\al+\beta;\al+\beta+1;-\frac{x}{L_W}
\Big),
\end{equation}
\end{subequations}
where ${}_2F_1$ is the (Gauss) hypergeometric function. These two terms
~\eqref{eq:bc31} and~\eqref{eq:bc32} take into account the major
contribution from (III).

One should notice that by choosing the limit of the integration in (I)
from $-L_W$ to $L_W$, we need the values of $u$ on the interval
$[-(L+L_W),L+L_W]$ instead of $[-L,L]$. The algebraic extension 
$u(y) \approx u(L)L^\beta |y|^{-\beta}$ as $y \to \infty$ and $u(y) \approx u(-L)L^\beta |y|^{-\beta}$ as $y\to -\infty$ can be used again to 
provide the information outside the interval $[-L,L]$. 
In the numerical experiments below, $L_W$ is chosen to be $2L$. 
When (I) is evaluated by a fast convolution algorithm, the total 
computational cost is $O(N\log N)$, where $N$ is the number of
grid points on the truncated computational domain $[-L,L]$.

\begin{figure}[htp]
 \begin{center}
  \includegraphics[totalheight=0.26\textheight]{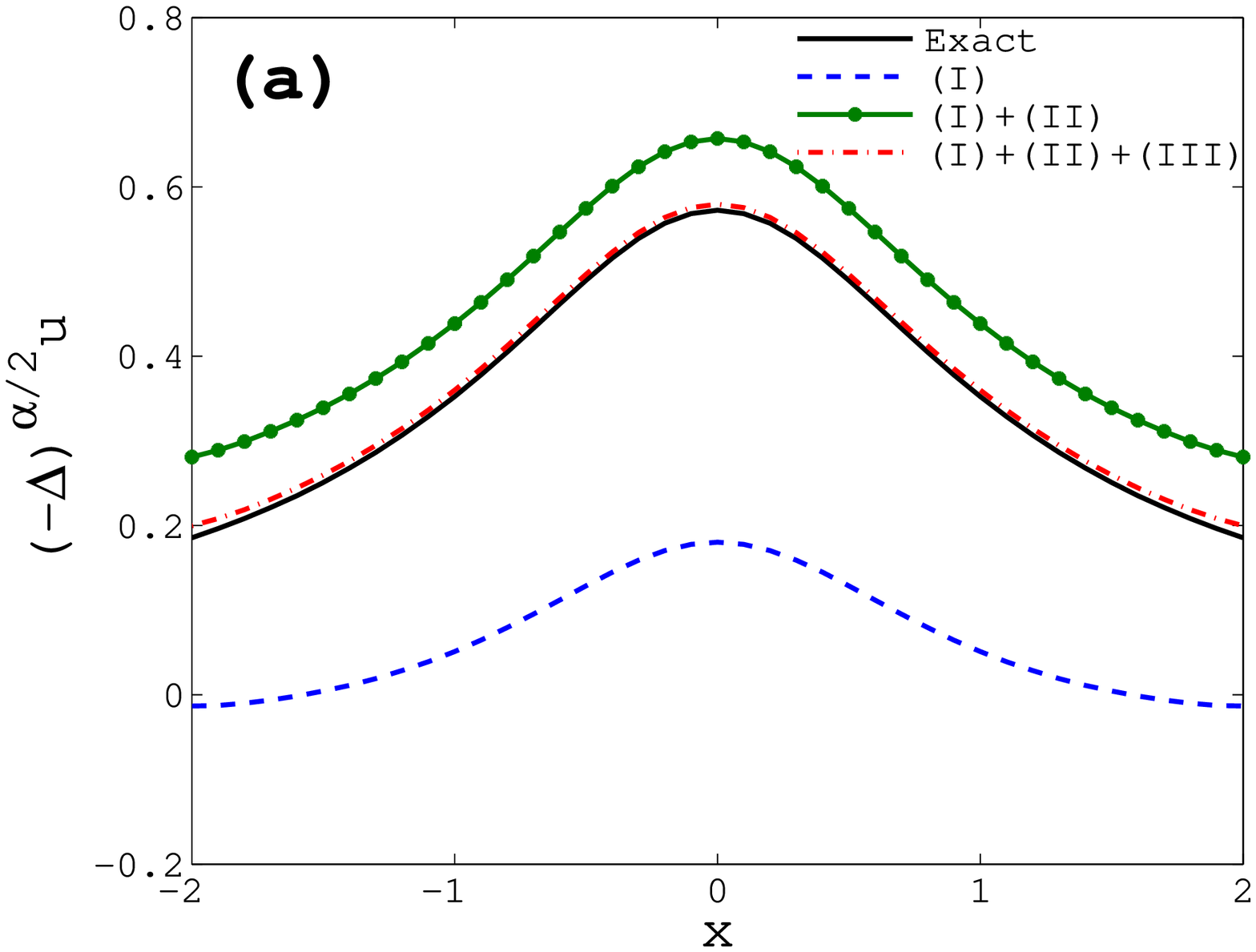}
$~~~~$
  \includegraphics[totalheight=0.26\textheight]{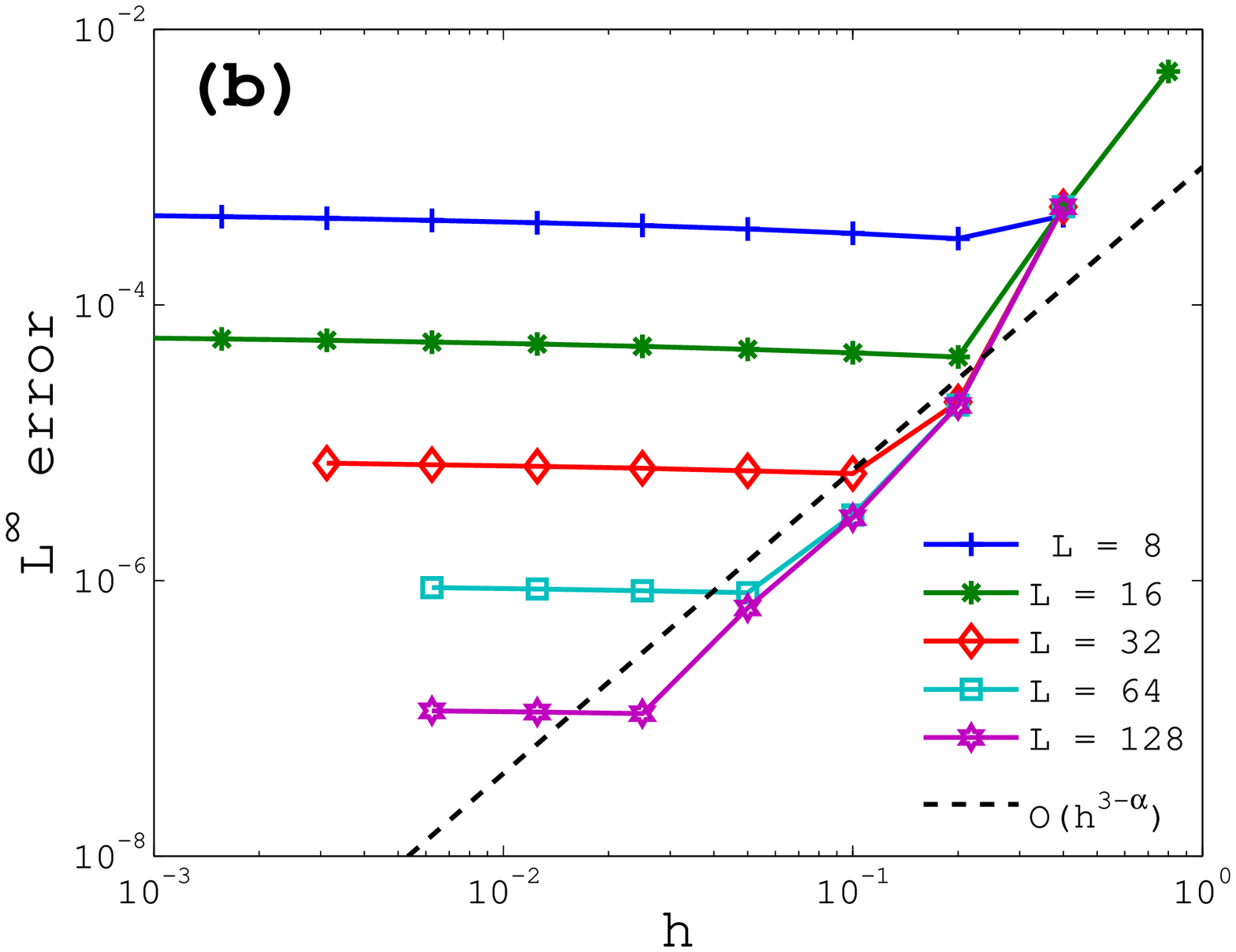}
 \end{center}
\caption{{\bf (a)}: The contributions of the terms (\textrm{I}), (\textrm{II}) and (\textrm{III}) to the accuracy of the 
method (using $w^Q$) for fractional Laplacian \eqref{eq:onedfractest}, with 
$\al = 0.4$, $L=2.0$ and $h=0.1$.  {\bf (b)}: The $L^\infty$ accuracy of the method using 
 different domain sizes and grid points.   The error saturates near the dashed line, which corresponds to error $O(h^{3-\alpha})$.
 }
\label{fig:convtest}
\end{figure}

The contributions of the three terms (I), (II) and (III) are illustrated
in Figure~\ref{fig:convtest}(a) using the weights $w^Q$, for 
\[
u(x) = (1+x^2)^{-(1-\al)/2}
\]
with the exact fractional Laplacian
\begin{equation} \label{eq:onedfractest}
\FL u(x)
=2^\al \Gamma\Big(\frac{1+\al}{2}\Big)\Gamma\Big(\frac{1-\al}{2}\Big)^{-1}(1+x^2)^{-(1+\al)/2}.
\end{equation}
When the computational domain is taken to be as small as $[-2,2]$, neither
(I) nor  (I)+(II) gives an accurate result. 
However, the inclusion of the 
contributions~\eqref{eq:bc31} and~\eqref{eq:bc32} from (III) reduces the overall error significantly, despite the crude approximation of 
an algebraic tail with the exact exponent $\beta=1-\alpha$. 

The accuracy (in the $L^\infty$ norm) of the operator 
as a function of the domain size, $L$, and the grid spacing,  $h$, is shown 
in Figure~\ref{fig:convtest}(b). One salient feature is the saturation of
the error: for a fixed size, $L$, of the domain, the 
error does not decrease any more when $h$ is smaller than a certain 
critical value. This gives another indication on the importance of the far 
field boundary conditions. The convergence rate before the saturation 
is $3-\alpha$, consistent with the error analysis summarized in Lemma~\ref{lem:accuracy1}.

\begin{remark} 
The method extends to other far field boundary conditions, 
such as $u(x) \to c_\pm$  as $x\to \pm \infty$ with an algebraic rate.  (This is the case for certain  traveling waves in fractional conservation laws).
\end{remark}

\begin{remark} When the exponent $\beta$ is not available, it could be estimated from the solution itself by data fitting, 
assuming that $L$ is large enough to be in the algebraic decaying region. 
\end{remark}

\section{Numerical experiments}
\label{sec:numerics}
In this section we perform additional  numerical experiments.   We perform two kinds of tests: the first is a test of the accuracy of the operator (measuring the error in applying the operator)  the second is a test of convergence (measuring the error in the solution of the equation involving the operator).

  The results validate the theoretically predicted accuracy and convergence rates for smooth functions.   These examples, which have slow decay in the far field, also illustrate the errors which come from uncontrolled errors in the truncated domain.

\subsection{Accuracy for smooth functions}\label{sec:num1}

The first computational example involves the function  $u(x)=e^{-x^2}$, which is  a smooth function with exponential decay. The fractional Laplacian of $u$ at  $x=0$ can be obtained exactly using the  Fourier Transform, 
\[
 (-\Delta)^{\alpha/2}u(0) = \frac{1}{\sqrt{\pi}}\int_0^\infty k^\alpha e^{-k^2/4} dk = 2^\alpha\Gamma\Big(\frac{1+\alpha}{2}\Big)\big/\sqrt{\pi}.
\]
Because $u$ decays exponentially, there is no need to consider the boundary 
contribution from (III). The convergence rate is shown in Figure~\ref{fig:CVcomp}.   In addition, in this figure, we compared our method  with the method from~\cite{cont2005finite} (for different values of $\epsilon$). In this context, the
scheme from~\cite{cont2005finite} reads
\begin{align*}
(-\Delta)^{\alpha/2}u(0) &= C_{1,\alpha}\int_{-\epsilon}^\epsilon 
\big( u(0)-u(y)\big)|y|^{-1-\alpha}dy 
+C_{1,\alpha}\int_{|y|> \epsilon} \big( u(0)-u(y)\big)|y|^{-1-\alpha}dy  \cr
&\approx -\frac{C_{1,\alpha}\epsilon^{2-\alpha}}{2-\alpha}u''(0)
+C_{1,\alpha} \sum_{j=0}^\infty \big(u(0)-u(\epsilon+(j+1/2)h)\big)\int_{\epsilon+jh}^{\epsilon+(j+1)h} |y|^{-1-\alpha}dy \cr
&\qquad +C_{1,\alpha} \sum_{j=0}^\infty \big(u(0)-u(-\epsilon-(j+1/2)h)\big)\int_{-\epsilon-(j+1)h}^{-\epsilon-jh} |y|^{-1-\alpha}dy, 
\end{align*}
where $u''(0)$ is also approximated by central difference.

The accuracy of the method from~\cite{cont2005finite}  is very close to $O(h^{2-\al})$, which is the accuracy of  our method using the weights $w^T$.  For our method using the weights $w^Q$, the accuracy is $O(h^{3-\al})$.   

 In addition, the measures in \cite{cont2005finite} decayed exponentially, so in our case, 
  the errors from  truncation of the integral in the far field are more significant.
  In addition,  the accuracy of these methods  is $O(h^{2-\al})$ which degenerates as $\al \to 2$.  For this reason, the quadratic weights $w^Q$ are needed to obtain reasonable accuracy.
 
The second example is a smooth function which decays algebraically: 
\[
u(x)=(1+x^2)^{-(1-\alpha)/2}.
\]
This solution was discussed in Section~\ref{sec:bciii}, and the fractional Laplacian of $u$ is given by~\eqref{eq:onedfractest}.
 In this case, the scheme recovers 
 the asymptotic error $O(h^{2-\alpha})$ for $w^T$ and $O(h^{3-\alpha})$ for $w^Q$, provided the computational domain is large enough to control the truncation error .  We performed computations with computational domain size $L = 8$ and $L=64$ to demonstrate the interaction of the domain size $L$ with the error.  See Figure~\ref{fig:smoothconv}. 

\begin{figure}[htp]
 \begin{center}
  \includegraphics[totalheight=0.26\textheight]{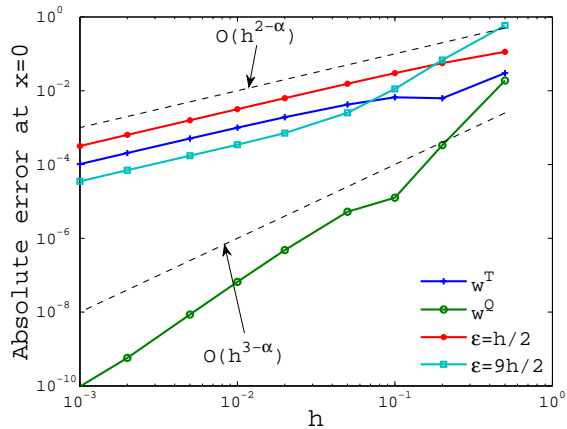}
 \end{center}
 \caption{The accuracy of the scheme on the function $u(x)=e^{-x^2}$.  The $w^T$ method and 
the method in~\cite{cont2005finite}, with two values of $\e$ are $O(h^{2-\alpha})$. 
The $w^Q$ method is the most accurate:  $O(h^{3-\alpha})$.
With parameters are $L=10$ and $\alpha=0.8$.
}
\label{fig:CVcomp}
\end{figure}

\begin{figure}[htp]
 \begin{center}
  \includegraphics[totalheight=0.26\textheight]{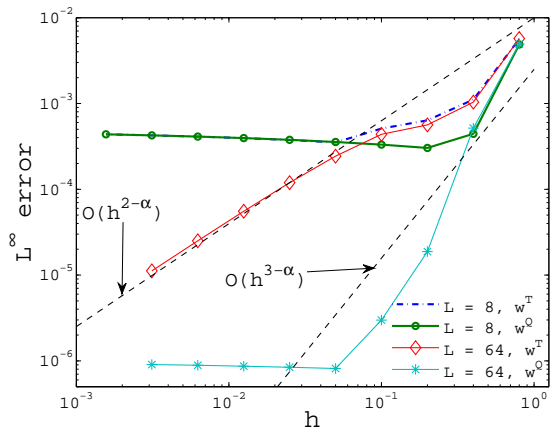}
 \end{center}
  \caption{
 The accuracy of the scheme on $u(x)=(1+x^2)^{-(1-\alpha)/2}$, with $\alpha = 0.8$, for two domain sizes.  The error for each rule is as predicted, but saturates for fixed domain sizes.
 }
  \label{fig:smoothconv} 
\end{figure}

\subsection{Accuracy for non-smooth functions}
When the function is non-smooth the method is still consistent, but the accuracy is decreased.
The next example is for a solution  $u$ is which is  continuous, but the derivatives at $x=\pm 1$  are discontinuous, so $u \in C^0(\R)$.
It comes from taking 
\[
f(x) =  (1-x^2)_+^{1-\al/2},
\]
then the solution of $\FL u = f$ is given by\footnote{The formula here is slightly different from the one in the reference, because we have corrected what we believe to be two typos}~\cite{MR2817383},
\begin{equation}\label{eq:explobs}
u(x)=\begin{cases}
      2^{-\al}\pi^{-1/2} \Gamma\left(\frac{1-\al}{2}\right)\Gamma\left(2-\frac{\al}{2}\right)
\Big(1-(1-\al)x^2\Big)
\qquad  &\mbox{if\ } |x|\leq 1,\cr
2^{-\al} \frac{ \Gamma\left(\frac{1-\al}{2}\right)\Gamma\left(2-\frac{\al}{2}\right)} 
{\Gamma\left(\frac{\al}{2}\right) \Gamma\left(\frac{5-\al}{2}\right)}|x|^{\al-1} {}_2F_1\left(
\frac{1-\al}{2},\frac{2-\al}{2};\frac{5-\al}{2};\frac{1}{|x|^2}
\right) &\mbox{if\ } |x|\geq 1.
     \end{cases}
\end{equation}

As shown in Figure~\ref{fig:obconvtest}(a), the error is largest near $x=\pm 1$, the points with discontinuous second order derivatives. 
Numerical experiments give an accuracy of  $O(h^{1-\al/2})$, see Figure~\ref{fig:obconvtest}(b). While consistent with the decreased regularity of the solution, this is  much slower than the previous (more regular) example.  

However, the accuracy improves as the solution becomes smoother. 
The next example involves a function,  $u$, which is  $C^1(\mathbb{R})$, with $u''$ discontinuous at $x=\pm 1$. 
This function arises when we take (note the exponent is larger)
\[
f(x)=(1-x^2)_+^{2-\alpha/2}.
\]
Then the solution of $\FL u = f$ is given by~\cite{MR2817383}
\begin{equation}\label{eq:explobs1}
u(x)=\begin{cases}
      2^{-\al-1}\pi^{-1/2} \Gamma\left(\frac{1-\al}{2}\right)\Gamma\left(3-\frac{\al}{2}\right)
\Big(1-(2-2\al)x^2+(1-\frac{4}{3}\alpha+\frac{1}{3}\alpha^2)x^4\Big)
\ &\mbox{if\ } |x|\leq 1,\cr
2^{-\al} \frac{ \Gamma\left(\frac{1-\al}{2}\right)\Gamma\left(3-\frac{\al}{2}\right)} 
{\Gamma\left(\frac{\al}{2}\right) \Gamma\left(\frac{7-\al}{2}\right)}|x|^{\al-1} {}_2F_1\left(
\frac{1-\al}{2},\frac{2-\al}{2};\frac{7-\al}{2};\frac{1}{|x|^2}
\right) &\mbox{if\ } |x|\geq 1.
     \end{cases}
\end{equation}
Taking $\alpha=0.8$ we computed the operator on $[-2,2]$ and plot the results in Figure~\ref{fig:obconvtest1}(a). 
For $w^T$, the convergence rate is  $O(h^{2-\alpha})$, as expected.
But for $w^Q$, it is $O(h^{2-\alpha/2})$, better than for $w^T$, but not as good as the rate $O(h^{3-\alpha})$ which requires more regularity of the solution.   Refer to Figure~\ref{fig:obconvtest1}(b).  
Again, the accuracy is high enough that the local truncation error is dominated by the error from the truncation of the far field boundary conditions, as $h$ is decreased, for fixed $L$.

\begin{figure}[htp]
 \begin{center}
\includegraphics[totalheight=0.26\textheight]{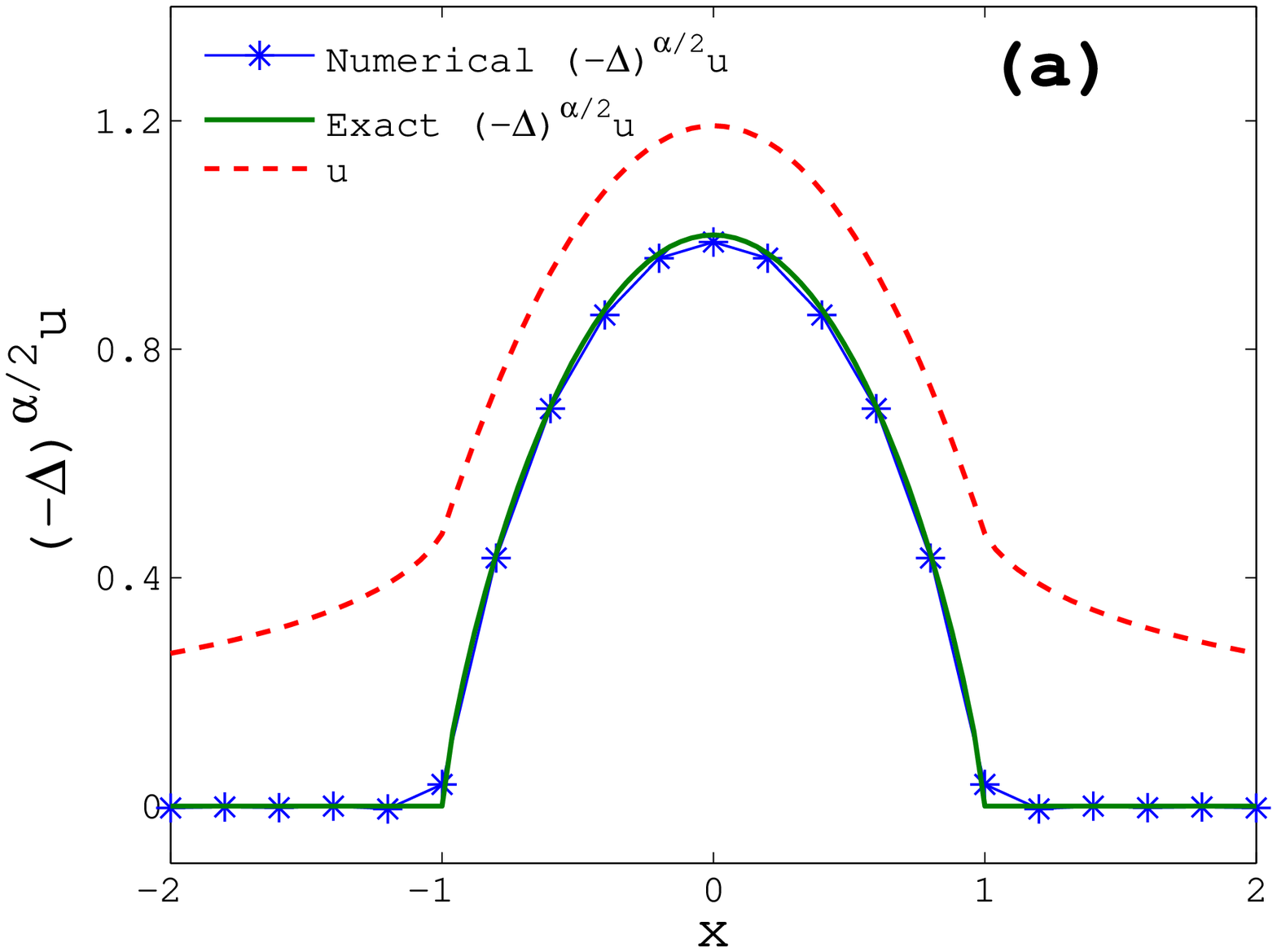}
$~~~~$
  \includegraphics[totalheight=0.26\textheight]{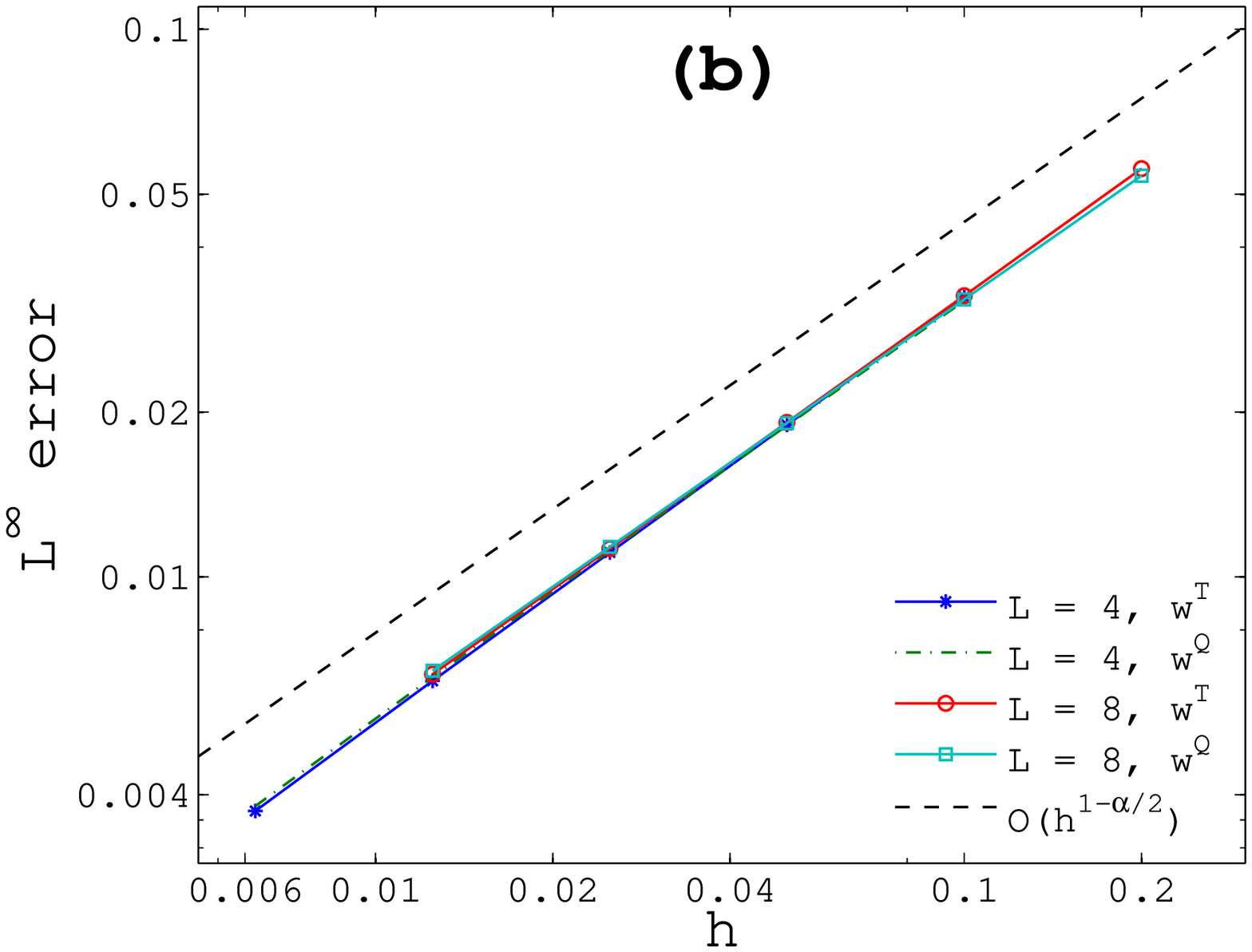}  
 \end{center}
\caption{
{\bf (a)} The fractional Laplacian of $u$ (dashed line) defined in~\eqref{eq:explobs} with $L=2$, $\al=0.4$ and $h=0.2$. {\bf (b)}
The $L^\infty$ error of $(-\Delta_h)^{\alpha/2}u$ for different grid sizes $h$, computed on $[-4,4]$ and $[-8,8]$. }
\label{fig:obconvtest}
\end{figure}

\begin{figure}[htp]
 \begin{center}
\includegraphics[totalheight=0.26\textheight]{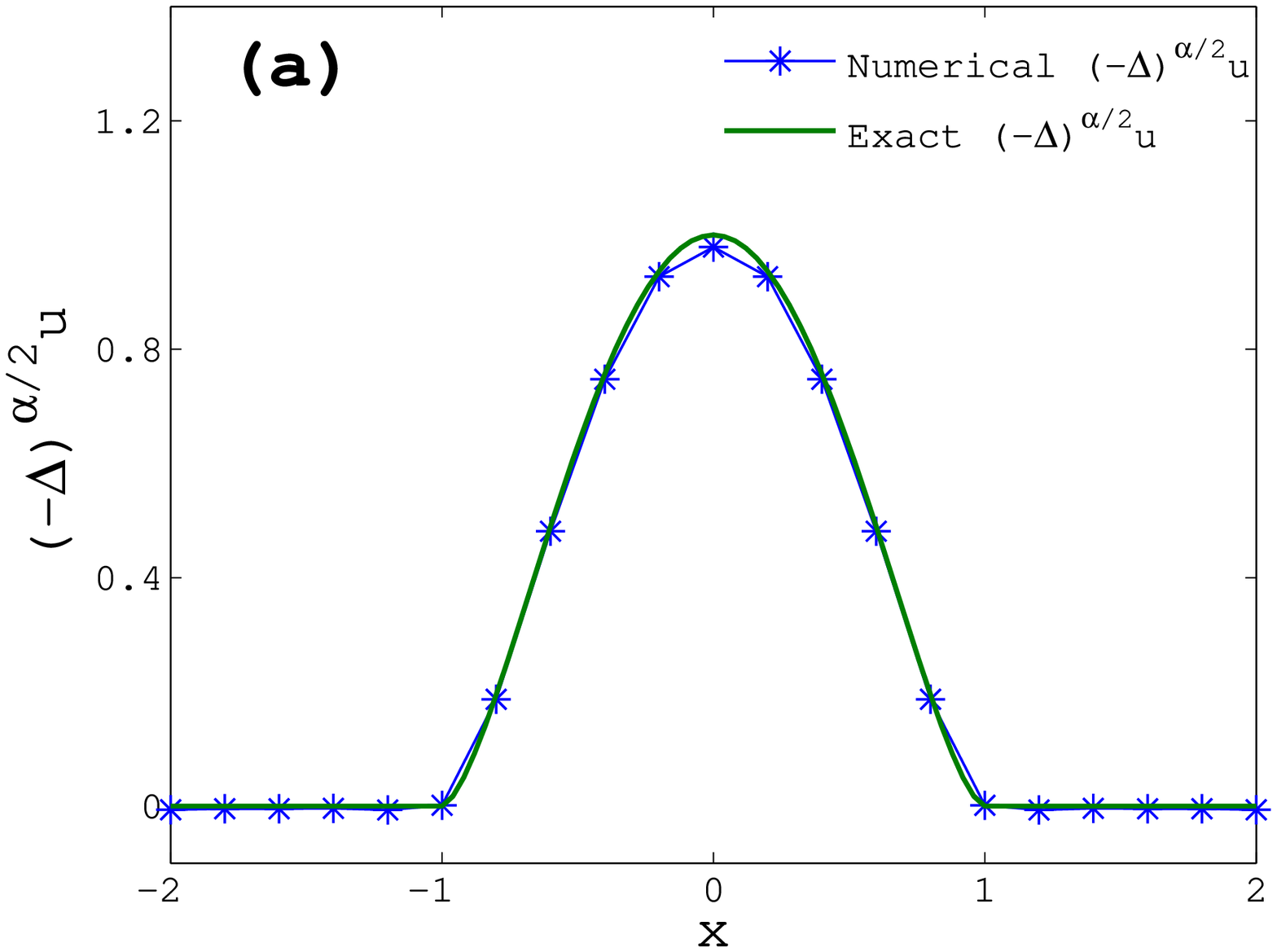}
$~~~~$
  \includegraphics[totalheight=0.26\textheight]{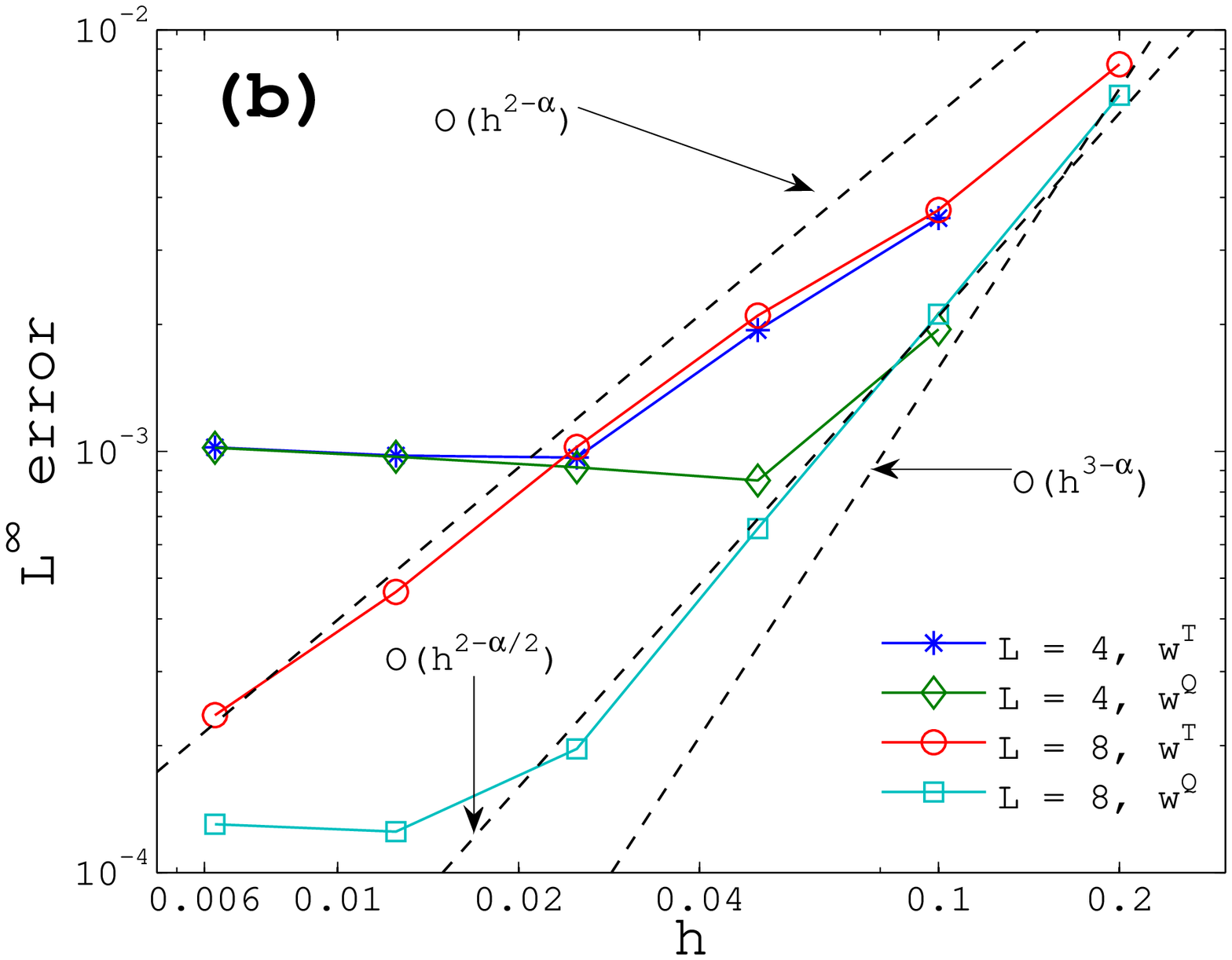}  
 \end{center}
\caption{
{\bf (a)} The fractional Laplacian of $u$ defined in~\eqref{eq:explobs1} with $L=2$, $\al=0.4$ and $h=0.2$. {\bf (b)}
The $L^\infty$ error of $(-\Delta_h)^{\alpha/2}u$ with different grid sizes $h$ on computed on $[-4,4]$ and $[-8,8]$. }
\label{fig:obconvtest1}
\end{figure}

\subsection{Extended Dirichlet problem} 
Next we solve the extended Dirichlet problem 
\begin{equation}\label{eq:GetoorDP}
 (-\Delta)^{\alpha/2} u=1\quad \mbox{on}\ (-1,1),\qquad\quad  u=0 \
\mbox{  on } (-\infty,-1]\cup [1,\infty).
\end{equation}
The exact solution is the probability density function of the expected first exit time of the 
symmetric $\alpha$-stable L\'{e}vy process from the interval $[-1,1]$. 
and is given by~\cite{MR0137148} 
\begin{equation}\label{eq:GetoorExtSol}
 u(x) = \frac{2^{-\alpha}\Gamma(\frac{1}{2})}{
\Gamma(1+\frac{\alpha}{2})\Gamma(\frac{1+\alpha}{2})}
(1-x^2)_+^{\frac{\alpha}{2}}.
\end{equation}
In this case, the solution is $C^{0,\al/2}([-1,1])$.
Despite the singularity, convergence
to the exact solution is observed. Because of the singularity
of the solution at the boundaries $x=\pm 1$, the numerical solutions using both weights $w^Q$ and 
$w^T$ converge at the same rate: $O(h^{\alpha/2})$, as shown 
in Figure~\ref{fig:GetoorSol}(b).   The numerical solutions 
(we set  $L=1$) for two grid sizes,  $h=0.20$, and $h=0.05$,
are shown in Figure~\ref{fig:GetoorSol}(a).   
(In this case, the solution was obtained by directly solving a (non-sparse) linear equation, rather than by iteration).
The resulting linear system is strictly diagonally dominant and is therefore well-conditioned.

\begin{figure}[htp]
 \begin{center}
  \includegraphics[totalheight=0.26\textheight]{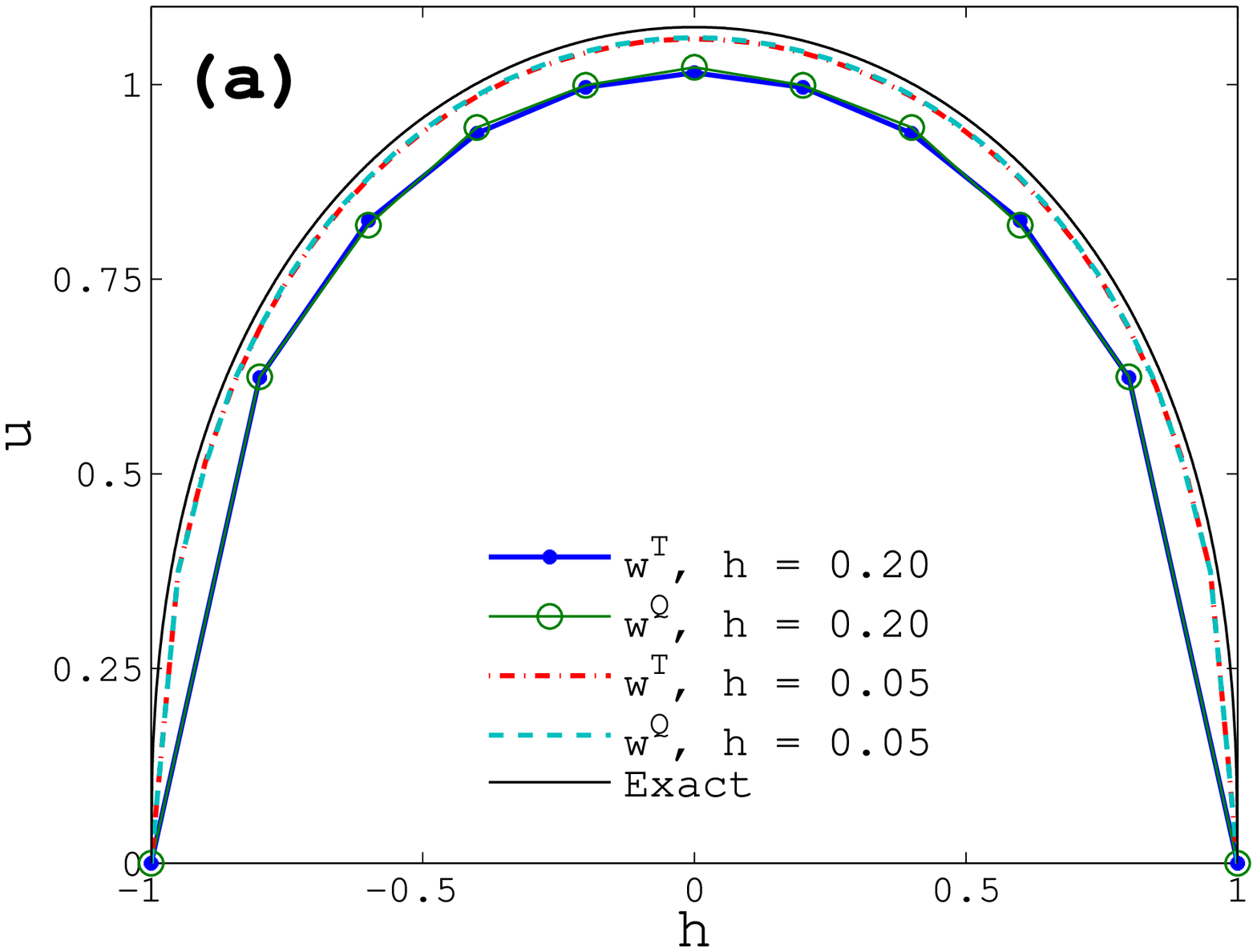}
$~~~$
  \includegraphics[totalheight=0.26\textheight]{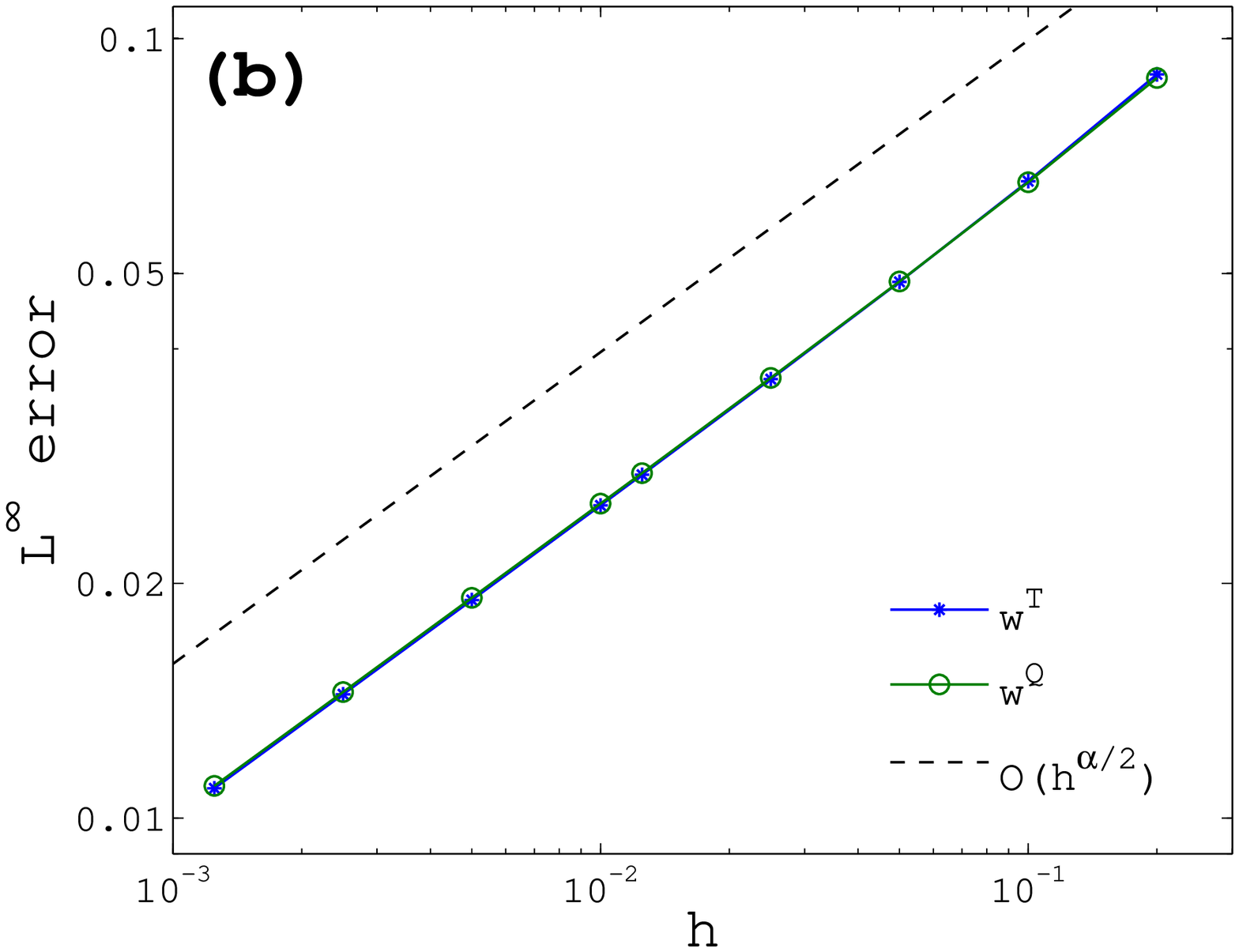}
 \end{center}
\caption{{\bf (a)} The numerical solution for the extended 
Dirichlet problem~\eqref{eq:GetoorDP} with parameters: $\alpha=0.8$,$L=1$,
$h=0.20$ or $0.05$; {\bf (b)} The convergence (in $h$) of the solution to the 
exact one~\eqref{eq:GetoorExtSol} for the weights $w^T$ and $w^Q$.}
\label{fig:GetoorSol}
\end{figure}

\subsection{The fractional obstacle problem}

The fractional obstacle problem~\cite{MR2367025,MR2244602,silvestre2007regularity} is a direct generalization of its 
elliptic counterpart: given a continuous function
$\varphi$ (the obstacle), consider the problem of determining a continuous function $u$ satisfying 
\begin{align}\label{eq:obstacle}
 u &\geq \varphi, && \mbox{in }\mathbb{R}^n,\cr
\FL  u &\geq 0,&& \mbox{in }\mathbb{R}^n, \\
\hskip 1in \FL u(x) &=0,&& \mbox{on } \{x\in\mathbb{R}^n \mid u(x)>\varphi(x)\}. 
\hskip 1in
\notag
\end{align}
The solution, $u(x)$, approaches zero as $|x| \to \infty$, which requires the obstacle function $\varphi$ 
to be  compactly supported or rapidly decaying.  

This problem arises  in the study of the long-time asymptotic behavior
of a fractional porous medium equation~\cite{MR2847534,MR2773189}.
Similar obstacle problems arise in financial mathematics as a pricing 
model~\cite{cont2004financial}.  

The equations~\eqref{eq:obstacle} can be written in the equivalent form
\[
 \min(u-\varphi,\FL u)=0,
\]
or as a steady state of the evolution equation 
\[
 u_t+\min(u-\varphi,\FL u)=0.
\]
The latter suggests the following iterative scheme
\begin{equation}\label{eq:obscheme}
 u^{k+1}_j =u^{k}_j- \Delta t \min\Big(u_j^k-\varphi_j,[\FLh u^k]_j\Big).
\end{equation}
When $\Delta t \leq \min\Big(1,\sum_{j\neq 0} w_j\Big)$, the discrete evolution operator $\mathcal{L}[u]$ is monotone in $u$, and the iterative method is a contraction in the maximum norm~\cite{ObermanSINUM}.
As a result, the extension to nonlinear elliptic operators involving the fractional Laplacian can be performed~\cite{biswas2010difference}, using the theory of viscosity solutions for these operators~\cite{alvarez1996viscosity}.   This allows us to conclude that the  scheme converges to the unique viscosity solution of the equation.
Moreover, if we start with the initial condition $u^0_j = \varphi_j$, then
$u^k_j$ is increasing in $k$ for each $j$ and converges to $u^\infty_j$ from below.

In one dimension, when the obstacle is $\varphi(x)=
2^{-\al}\pi^{-1/2} \Gamma\left(\frac{1-\al}{2}\right)\Gamma\left(\frac{4-\al}{2}\right)
\Big(1-(1-\al)x^2\Big)_+, \al\in (0,1)$, 
one can check that the exact solution $u$ is given by~\eqref{eq:explobs}, 
with the coincident set $[-1,1]$ (on which $u=\phi$).

The convergence of the scheme~\eqref{eq:obscheme} with $\al=0.5$ and $L=4$ is shown in Figure~\ref{fig:obsc}, for both the solution $u$ and the fractional Laplacian $(-\Delta)^{\al/2}u$. 
The $L^\infty$ errors on different domain sizes $L$ and grid sizes are shown in Figure~\ref{fig:obserr}, which is shown to be $O(h^{1+\frac{\alpha}{2}})$
for the solution, $u$, and $O(h^{1-\frac{\alpha}{2}})$
for the fractional Laplacian, $(-\Delta)^{\alpha/2}u$.
Note that the accuracy of the solution is better than the accuracy of the operator.
As in the example of  Figure~\ref{fig:convtest}, for fixed 
domain size $L$, the error becomes saturated as $h\to 0$.

\begin{figure}[htp]
 \begin{center}
  \includegraphics[totalheight=0.26\textheight]{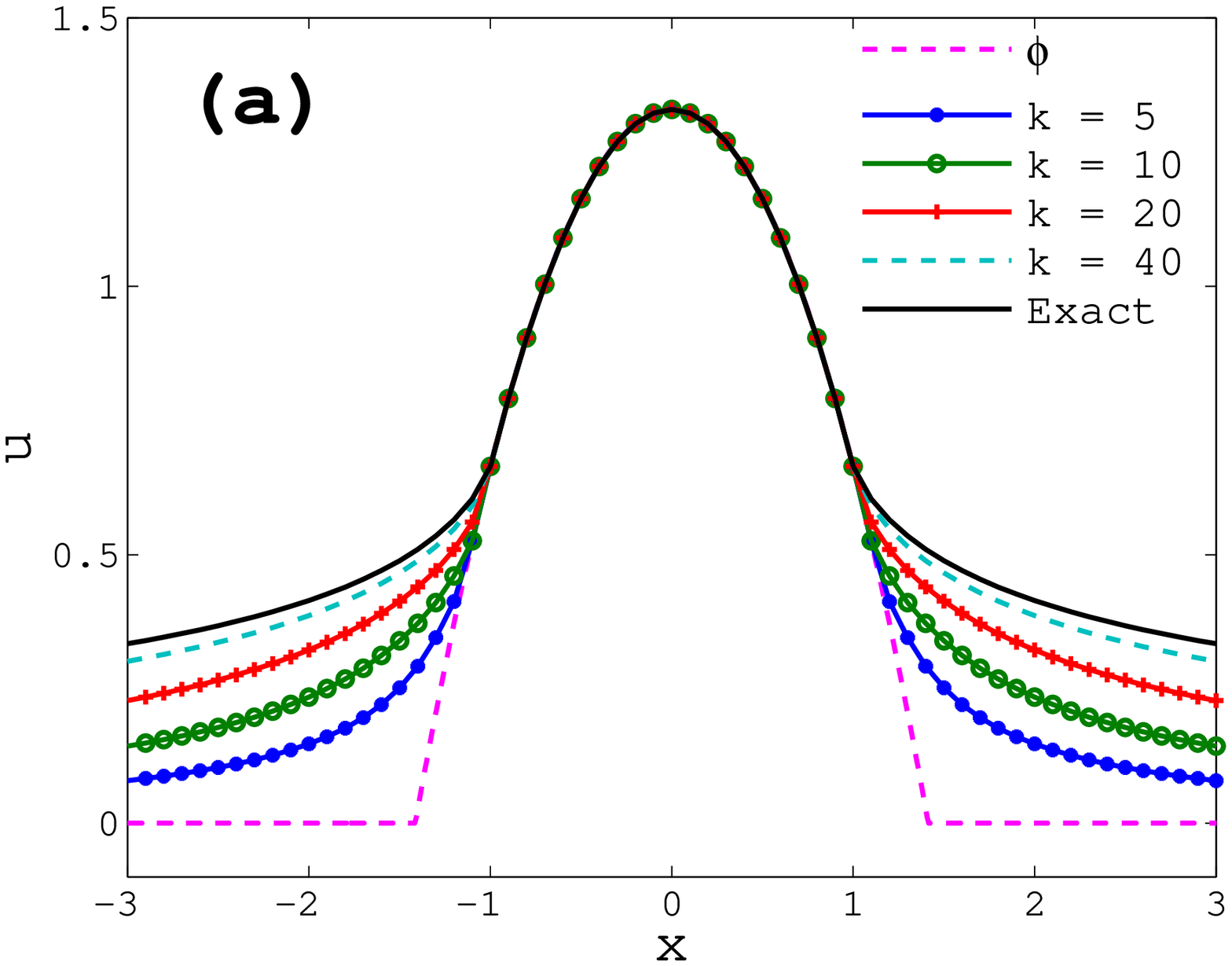}$~\qquad~$
\includegraphics[totalheight=0.26\textheight]{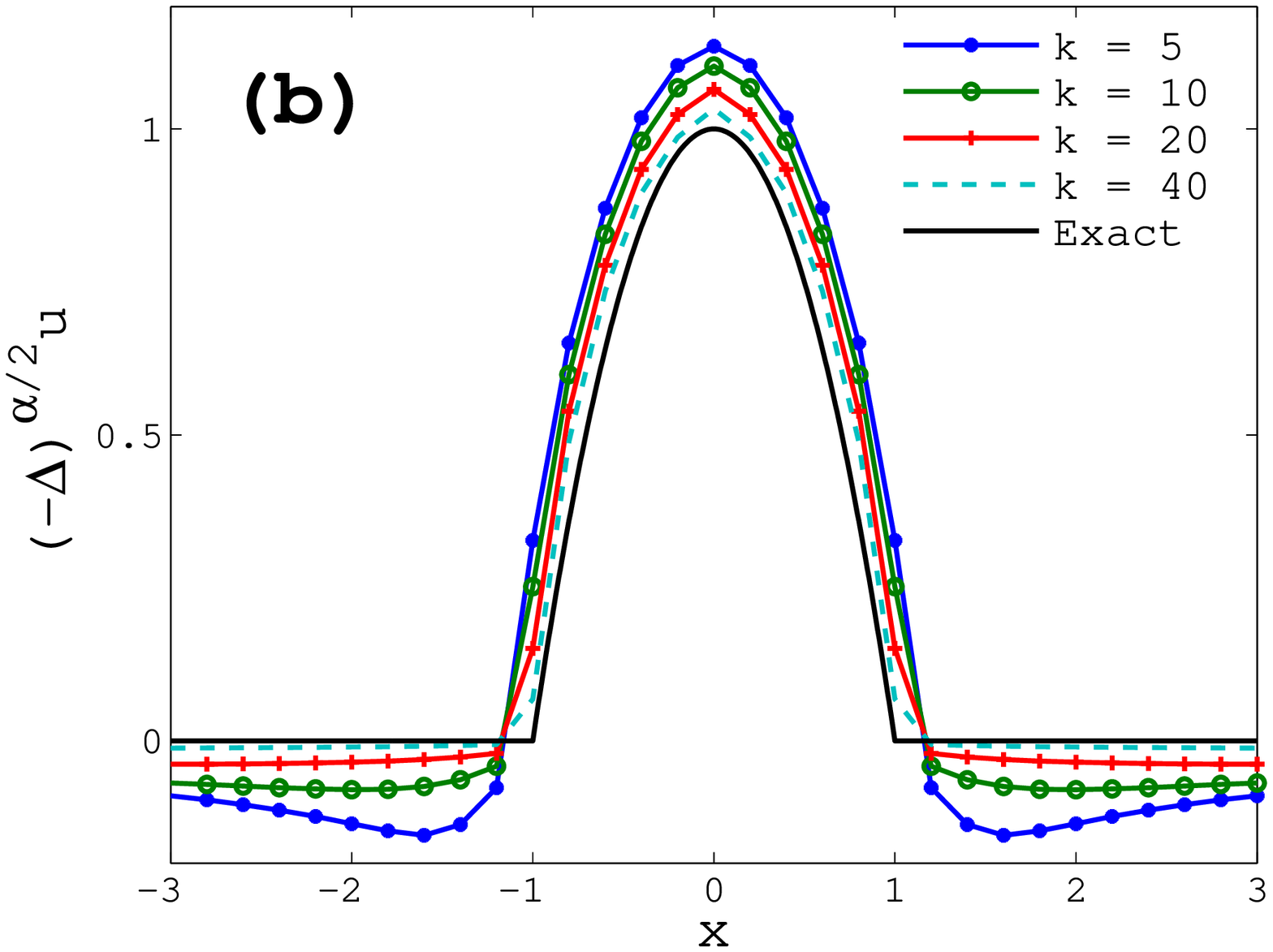}
 \end{center}
\caption{The convergence of the iterative scheme~\eqref{eq:obscheme} 
for the obstacle problem: {\bf (a)} comparing the numerical solution to the exact solution (given in \eqref{eq:explobs}) {\bf (b)} 
 the discrete fractional Laplacian $(-\Delta_h)^{\al/2}u_i$ to $(1-x^2)^{1-\alpha/2}_+$. The parameters are $\al=0.5$, $L=4$, $h=0.1$ and $\Delta t  = 0.158= h^\al/2$. }
\label{fig:obsc}
\end{figure}

\begin{figure}[htp]
 \begin{center}
\includegraphics[totalheight=0.26\textheight]{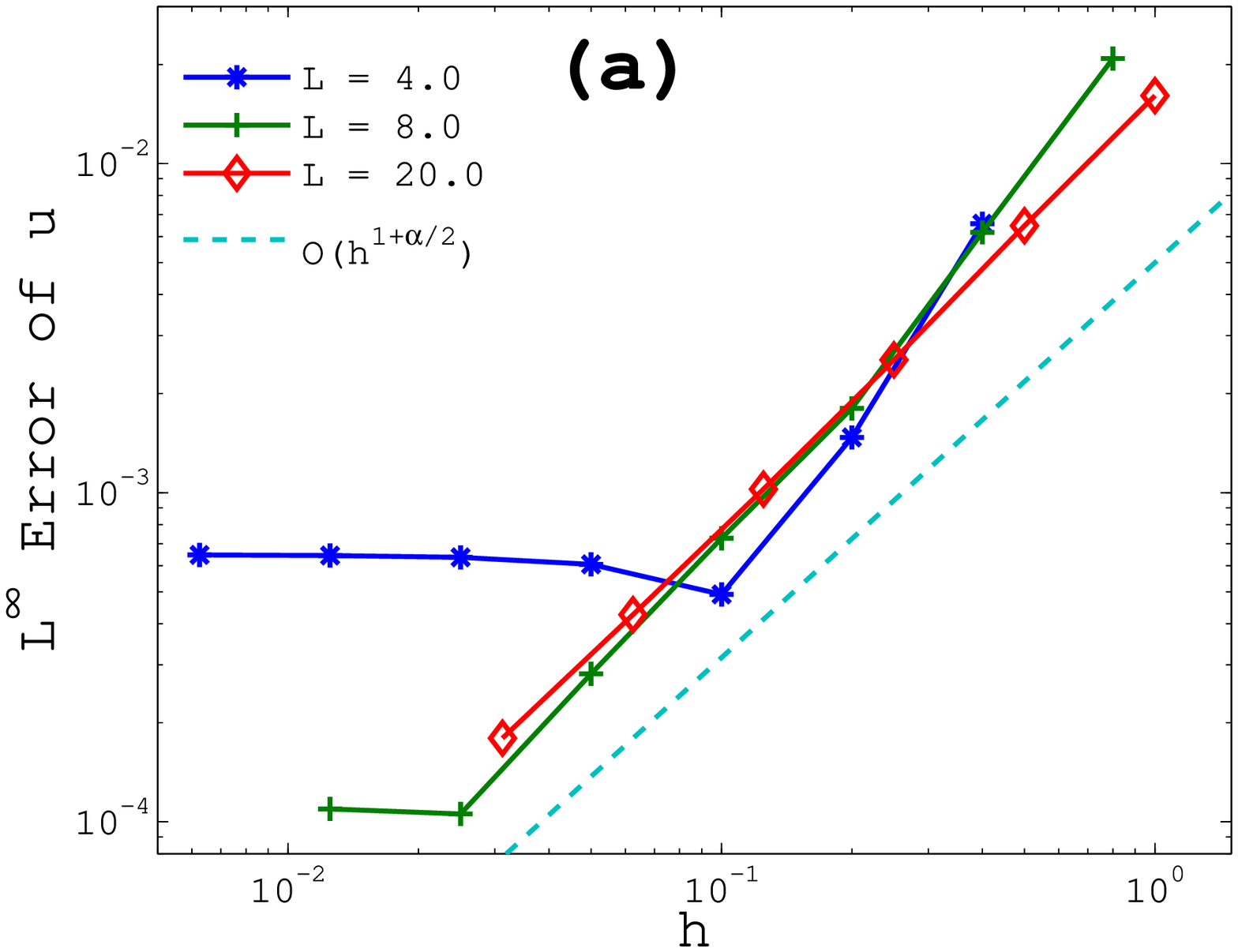}
$~\qquad~$
\includegraphics[totalheight=0.26\textheight]{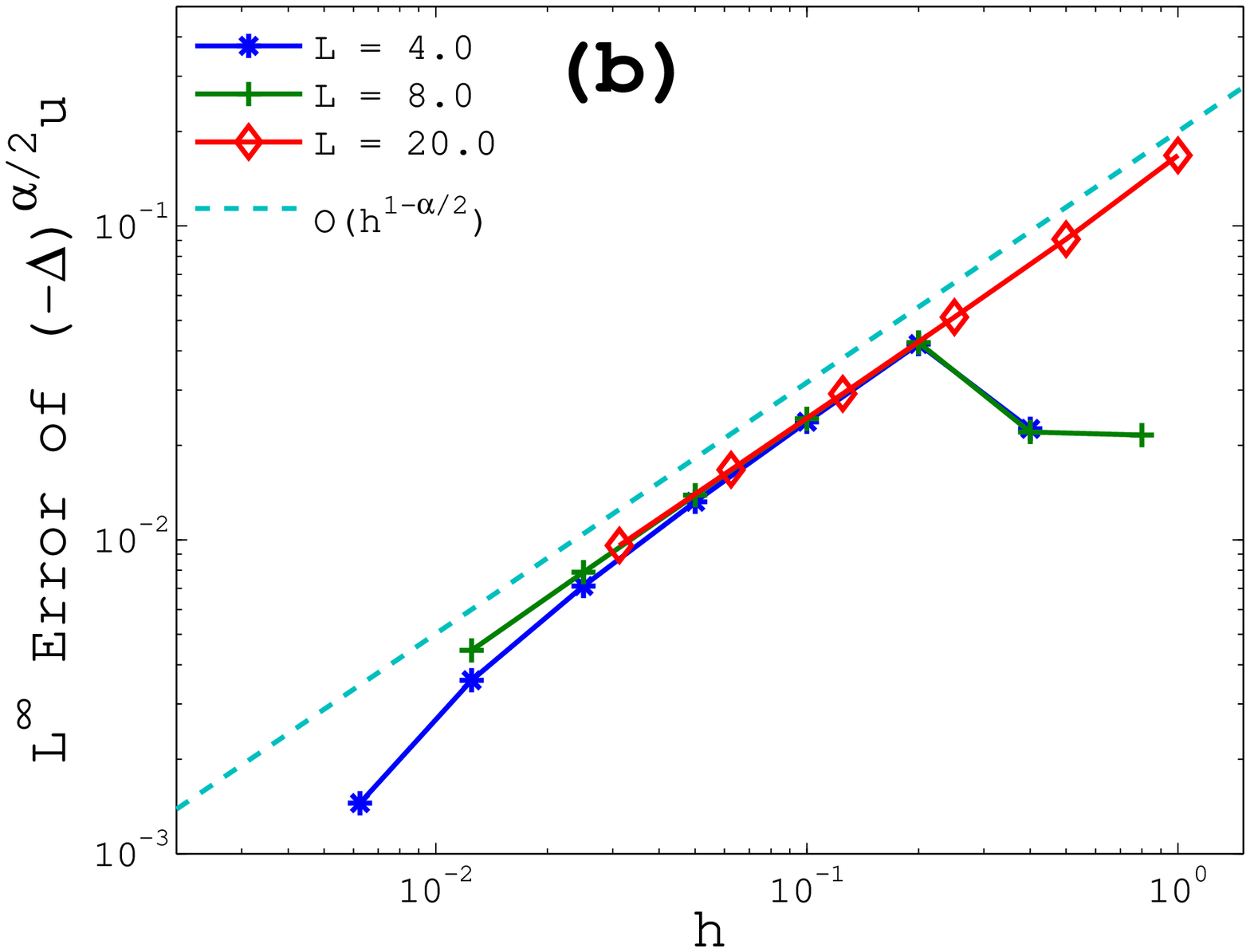}
 \end{center}
\caption{The convergence of the obstacle problem in $L^\infty$ norm
with different domain sizes $L$ and grid sizes $h$.  
{\bf (a)} the error in the solution $u$,  {\bf (b)} the error in the fractional Laplacian $(-\Delta_h)^{\al/2}u$.}
\label{fig:obserr}
\end{figure}

\section{Conclusions} 
We used the singular integral representation of the operator~\eqref{eq:rieszfrac} to derive a finite difference/quadrature discretization of the fractional Laplacian in one dimension.   
The weights in the discrete scheme~\eqref{FLh} are obtained from 
approximations of~\eqref{eq:rieszfrac} by splitting it into two parts: one coming from the singular part of the integral, the other from the tail of the integral. 

Two different sets of weights were obtained using semi-exact quadrature in the tail of the integral.  The weights $w^T$ come from using linear interpolation functions, and the weights  $w^Q$ come from quadratic interpolation functions. 
The formal accuracy of the two schemes was found to be $O(h^{2-\alpha})$ for $w^T$ and $O(h^{3-\alpha})$ for $w^Q$.   (The weights $w^T$ were included mainly because they are simpler to derive).
 
The discrete scheme obtained from the weights shares important properties with the continuous singular integral operator.  The weights are positive, and they scale as expected in the parameter $h$ and in the distance $x_i = hi$. 
The scheme is consistent  in the limit  $\alpha \to 2^-$, in the sense that it recovers the standard centered finite different approximation for the Laplacian in one dimension.
 A formula for the sum of the weights was given, which provides a CFL condition for stability of the explicit (e.g.\ forward Euler) discretizations of time dependent problem involving the operator.

A practical issue is that the operator is posed on the entire line, and truncation of the computational domain can dominate the error.  For solutions which decay only algebraically, this requires the use of large computational domains.
We developed an approximation for the truncated operator on finite domain,  using asymptotic values of the extended Dirichlet data.  This led to an improvement of the accuracy by  an order of magnitude.

We measured the accuracy of the operator using known solutions.  Smooth solutions to validated the predicted convergence rates.  Several exact solutions were computed: exponentially decaying solutions could be computed on small domains.  Algebraically decaying solution required larger domains, even with the approximation of the far field boundary condition.  We also computed the singular Getoor solution, both on the line and for the extended Dirichlet problem on the interval.  Despite the singularity, the numerical solutions converge, but with a slower rate, $O(h^{\alpha/2})$.

Once equipped with a consistent, stable  scheme with positive weights for the Fractional Laplacian, we can compute viscosity solutions of nonlinear elliptic and parabolic PDEs involving the operator.  We computed solutions of the obstacle problem using a simple iterative method.  Solutions were obtained, and the computed convergence rate was $O(h^{1+\frac{\alpha}{2}})$.
  
Convergence of solutions to the extended Dirichlet problem was proved in Theorem~\ref{thm1}.  For  smooth solutions, the error (in the maximum norm)  of the solution is bounded by the truncation error.    
It is an open problem to build numerical methods which obtain higher accuracy for singular solutions.
The theory of  nonlocal viscosity solutions~\cite{alvarez1996viscosity}  can be applied to study weak solutions of the fractional laplacian (while theory was developed for the nonlinear case, is also applies to weak solutions in the linear case).   Using this theory, convergence rates for 
finite difference quadrature schemes for weak solutions can be obtained~\cite{biswas2008error}: in this case, the convergence rate is fractional with an unknown exponent, consistent with the results we obtained here.   The viscosity solutions approach has not been used to obtain the higher order rates we established here for the special case of smooth solutions, since it applies to weak solutions.  
Conversely, the analogue of Theorem~\ref{thm1} could be proposed for weak solutions: however it is an open problem how to make these  techniques rigorous in for singular solutions

We compared our method with the method from  \cite{cont2005finite}, which was found to be less accurate than using $w^Q$ (but similar to using $w^T$.).  
In the supplementary materials, excerpted from~\cite{HuangOberman2}, an extensive comparison is performed with other numerical methods, which can also  be written in the form~\eqref{FLh}.   In particular, we make comparisons with schemes using fractional derivatives and Fourier based methods.  The scheme with quadratic weights is the most accurate of the stable schemes.  (The spectral method can be very accurate, but it is  unstable for $\al > 1$.)

Extending this work to higher dimensions still remains a challenge.  The simplest part of the extension would be the treatment of the singularity in the integral: this extends naturally (using polar or spherical coordinates) to higher dimensions.    The exact numerical quadrature can be performed in higher dimensions, although it becomes more complicated.  However, the treatment of the tail of the operator does not extend.  In the special case of operators with truncated tails, or for exponentially decaying  extended boundary conditions, the extension of this work is possible.

\bibliography{biofraclap}

\newcommand{\etalchar}[1]{$^{#1}$}
\begin{thebibliography}{dPQRV11}

\bibitem[App09]{MR2512800}
D.~Applebaum.
\newblock {\em L\'evy processes and stochastic calculus}, volume 116 of {\em
  Cambridge Studies in Advanced Mathematics}.
\newblock Cambridge University Press, Cambridge, second edition, 2009.

\bibitem[AT96]{alvarez1996viscosity}
O.~Alvarez and A.~Tourin.
\newblock Viscosity solutions of nonlinear integro-differential equations.
\newblock In {\em Annales de l'Institut Henri Poincar{\'e}. Analyse non
  lin{\'e}aire}, volume~13, pages 293--317. Elsevier, 1996.

\bibitem[BFW98]{MR1637513}
P.~Biler, T.~Funaki, and W.~A. Woyczynski.
\newblock Fractal {B}urgers equations.
\newblock {\em J. Differential Equations}, 148(1):9--46, 1998.

\bibitem[BIK11]{MR2817383}
P.~Biler, C.~Imbert, and G.~Karch.
\newblock Barenblatt profiles for a nonlocal porous medium equation.
\newblock {\em C. R. Math. Acad. Sci. Paris}, 349(11-12):641--645, 2011.

\bibitem[BJK08]{biswas2008error}
Imran~H Biswas, Espen~R Jakobsen, and Kenneth~H Karlsen.
\newblock Error estimates for a class of finite difference-quadrature schemes
  for fully nonlinear degenerate parabolic integro-pdes.
\newblock {\em Journal of Hyperbolic Differential Equations}, 5(01):187--219,
  2008.

\bibitem[BJK10]{biswas2010difference}
I.~H. Biswas, E.~R. Jakobsen, and K.~H. Karlsen.
\newblock Difference-quadrature schemes for nonlinear degenerate parabolic
  integro-{PDE}.
\newblock {\em SIAM J. Numer. Anal.}, 48(3):1110--1135, 2010.

\bibitem[BL13]{BurchLehoucq}
N.~Burch and R.~B. Lehoucq.
\newblock Computing the exit-time for a finite-range symmetric jump process.
\newblock Technical Report SAND 2013-2354J, Sandia National Labs, 2013.

\bibitem[CJ11]{MR2795714}
S.~Cifani and E.~R. Jakobsen.
\newblock Entropy solution theory for fractional degenerate
  convection-diffusion equations.
\newblock {\em Ann. Inst. H. Poincar\'e Anal. Non Lin\'eaire}, 28(3):413--441,
  2011.

\bibitem[CJK11a]{MR2832791}
S.~Cifani, E.~R. Jakobsen, and K.~H. Karlsen.
\newblock The discontinuous {G}alerkin method for fractal conservation laws.
\newblock {\em IMA J. Numer. Anal.}, 31(3):1090--1122, 2011.

\bibitem[CJK11b]{MR2855429}
S.~Cifani, E.~R. Jakobsen, and K.~H. Karlsen.
\newblock The discontinuous {G}alerkin method for fractional degenerate
  convection-diffusion equations.
\newblock {\em BIT}, 51(4):809--844, 2011.

\bibitem[CMG{\etalchar{+}}03]{chechkin2003first}
A.~V. Chechkin, R.~Metzler, V.~Y. Gonchar, J.~Klafter, and L.~V. Tanatarov.
\newblock First passage and arrival time densities for l{\'e}vy flights and the
  failure of the method of images.
\newblock {\em Journal of Physics A: Mathematical and General}, 36(41):L537,
  2003.

\bibitem[CS07]{caffarelli2007extension}
L.~Caffarelli and L.~Silvestre.
\newblock An extension problem related to the fractional laplacian.
\newblock {\em Communications in partial differential equations},
  32(8):1245--1260, 2007.

\bibitem[CSS08]{MR2367025}
L.~A. Caffarelli, S.~Salsa, and L.~Silvestre.
\newblock Regularity estimates for the solution and the free boundary of the
  obstacle problem for the fractional {L}aplacian.
\newblock {\em Invent. Math.}, 171(2):425--461, 2008.

\bibitem[CT04]{cont2004financial}
R.~Cont and P.~Tankov.
\newblock {\em Financial modelling with jump processes}, volume 133 of {\em
  Chapman \& Hall/CRC Financial Mathematics Series}.
\newblock Chapman \& Hall/CRC, Boca Raton, FL, 2004.

\bibitem[CV05]{cont2005finite}
R.~Cont and E.~Voltchkova.
\newblock A finite difference scheme for option pricing in jump diffusion and
  exponential l{\'e}vy models.
\newblock {\em SIAM Journal on Numerical Analysis}, 43(4):1596--1626, 2005.

\bibitem[CV11a]{MR2847534}
L.~Caffarelli and J.~L. V{\'a}zquez.
\newblock Nonlinear porous medium flow with fractional potential pressure.
\newblock {\em Arch. Ration. Mech. Anal.}, 202(2):537--565, 2011.

\bibitem[CV11b]{MR2773189}
L.~A. Caffarelli and J.~L. V{\'a}zquez.
\newblock Asymptotic behaviour of a porous medium equation with fractional
  diffusion.
\newblock {\em Discrete Contin. Dyn. Syst.}, 29(4):1393--1404, 2011.

\bibitem[Das11]{MR2807926}
S.~Das.
\newblock {\em Functional fractional calculus}.
\newblock Springer-Verlag, Berlin, second edition, 2011.

\bibitem[DFFL05]{diethelm2005algorithms}
K.~Diethelm, N.~J. Ford, A.~D. Freed, and Y.~Luchko.
\newblock Algorithms for the fractional calculus: a selection of numerical
  methods.
\newblock {\em Comput. Methods Appl. Mech. Engrg.}, 194(6):743--773, 2005.

\bibitem[DG13]{d2013fractional}
M.~D'Elia and M.~Gunzburger.
\newblock The fractional laplacian operator on bounded domains as a special
  case of the nonlocal diffusion operator.
\newblock {\em Computers \& Mathematics with Applications}, 66(7):1245--1260,
  2013.

\bibitem[DGLZ12]{du2012analysis}
Q.~Du, M.~Gunzburger, R.~B. Lehoucq, and K.~Zhou.
\newblock Analysis and approximation of nonlocal diffusion problems with volume
  constraints.
\newblock {\em SIAM review}, 54(4):667--696, 2012.

\bibitem[dPQRV11]{MR2737788}
A.~de~Pablo, F.~Quir{\'o}s, A.~Rodr{\'{\i}}guez, and J.~L. V{\'a}zquez.
\newblock A fractional porous medium equation.
\newblock {\em Adv. Math.}, 226(2):1378--1409, 2011.

\bibitem[Dro10]{MR2552219}
J.~Droniou.
\newblock A numerical method for fractal conservation laws.
\newblock {\em Math. Comp.}, 79(269):95--124, 2010.

\bibitem[dT13]{Felix}
F.~del Teso.
\newblock Finite difference method for a fractional porous medium equation.
\newblock {\em Calcolo}, pages 1--24, 2013.

\bibitem[dTV13]{JLFelix}
F.~del Teso and J.~L. V\'{a}zquez.
\newblock Finite difference method for a general fractional porous medium
  equation.
\newblock {\tt arXiv:1307.2474}, 2013.

\bibitem[Get61]{MR0137148}
R.~K. Getoor.
\newblock First passage times for symmetric stable processes in space.
\newblock {\em Trans. Amer. Math. Soc.}, 101:75--90, 1961.

\bibitem[GM98]{gorenflo1998random}
R.~Gorenflo and F.~Mainardi.
\newblock Random walk models for space-fractional diffusion processes.
\newblock {\em Fract. Calc. Appl. Anal.}, 1(2):167--191, 1998.

\bibitem[Her11]{herrmann2011fractional}
R.~Herrmann.
\newblock {\em Fractional calculus: an introduction for physicists}.
\newblock World Scientific Pub Co Inc, Singapore, 2011.

\bibitem[Hil00]{MR1890104}
R.~Hilfer.
\newblock {\em Applications of fractional calculus in physics}.
\newblock World Scientific Publishing Co. Inc., River Edge, NJ, 2000.

\bibitem[HO]{HuangOberman2}
Y.~Huang and A.~M. Oberman.
\newblock Numerical methods for the fractional {L}aplacian: a comparative
  study.
\newblock in preparation.

\bibitem[Lan72]{MR0350027}
N.~S. Landkof.
\newblock {\em Foundations of modern potential theory}.
\newblock Springer-Verlag, New York, 1972.
\newblock Die Grundlehren der mathematischen Wissenschaften, Band 180.

\bibitem[LT09]{larsson2009partial}
S.~Larsson and V.~Thom{\'e}e.
\newblock {\em Partial differential equations with numerical methods},
  volume~45.
\newblock Springer, 2009.

\bibitem[NOS13]{nochetto2013pde}
R.~H. Nochetto, E.~Otarola, and A.~J. Salgado.
\newblock A pde approach to fractional diffusion in general domains: a priori
  error analysis.
\newblock {\em arXiv preprint arXiv:1302.0698}, 2013.

\bibitem[Obe06]{ObermanSINUM}
A.~M. Oberman.
\newblock Convergent difference schemes for degenerate elliptic and parabolic
  equations: {H}amilton-{J}acobi equations and free boundary problems.
\newblock {\em SIAM J. Numer. Anal.}, 44(2):879--895 (electronic), 2006.

\bibitem[OS74]{MR0361633}
K.~B. Oldham and J.~Spanier.
\newblock {\em The fractional calculus: Theory and applications of
  differentiation and integration to arbitrary order}.
\newblock Academic Press, New York-London, 1974.

\bibitem[Rai00]{raible2000levy}
S.~Raible.
\newblock {\em L{\'e}vy processes in finance: Theory, numerics, and empirical
  facts}.
\newblock PhD thesis, Universit{\"a}t Freiburg i. Br, 2000.

\bibitem[Sil06]{MR2244602}
L.~Silvestre.
\newblock H\"older estimates for solutions of integro-differential equations
  like the fractional {L}aplace.
\newblock {\em Indiana Univ. Math. J.}, 55(3):1155--1174, 2006.

\bibitem[Sil07]{silvestre2007regularity}
L.~Silvestre.
\newblock Regularity of the obstacle problem for a fractional power of the
  laplace operator.
\newblock {\em Comm. Pure Appl. Math.}, 60(1):67--112, 2007.

\bibitem[SKM93]{MR1347689}
S.~G. Samko, A.~A. Kilbas, and O.~I. Marichev.
\newblock {\em Fractional integrals and derivatives: Theory and applications}.
\newblock Gordon and Breach Science Publishers, Yverdon, 1993.

\bibitem[Ste70]{MR0290095}
E.~M. Stein.
\newblock {\em Singular integrals and differentiability properties of
  functions}.
\newblock Princeton Mathematical Series, No. 30. Princeton University Press,
  Princeton, N. J., 1970.

\bibitem[TD13]{tian2013analysis}
X.~Tian and Q.~Du.
\newblock Analysis and comparison of different approximations to nonlocal
  diffusion and linear peridynamic equations.
\newblock {\em SIAM Journal on Numerical Analysis}, 51(6):3458--3482, 2013.

\bibitem[TMS06]{Tadjeran2006205}
C.~Tadjeran, M.~M. Meerschaert, and H.-P. Scheffler.
\newblock A second-order accurate numerical approximation for the fractional
  diffusion equation.
\newblock {\em J. Comput. Phys.}, 213(1):205 -- 213, 2006.

\bibitem[TW06]{MR2227237}
A.~Truman and J.-L. Wu.
\newblock Fractal {B}urgers' equation driven by {L}\'evy noise.
\newblock In {\em Stochastic partial differential equations and
  applications---{VII}}, volume 245 of {\em Lect. Notes Pure Appl. Math.},
  pages 295--310. Chapman \& Hall/CRC, Boca Raton, FL, 2006.

\bibitem[Val09]{Vald11}
E.~Valdinoci.
\newblock From the long jump random walk to the fractional {L}aplacian.
\newblock {\em Bol. Soc. Esp. Mat. Apl. SeMA}, (49):33--44, 2009.

\bibitem[ZRK07]{FLBD}
A.~Zoia, A.~Rosso, and M.~Kardar.
\newblock {F}ractional {L}aplacian in bounded domains.
\newblock {\em Phys. Rev. E}, 76:021116, Aug 2007.

\end{thebibliography}
\bibliographystyle{alpha}
\end{document}